\documentclass[12pt]{amsart}

\usepackage{amsmath,amsfonts,bm}

\def\eqref#1{equation~\ref{#1}}

\def\1{\bm{1}}

\DeclareMathAlphabet{\mathsfit}{\encodingdefault}{\sfdefault}{m}{sl}
\SetMathAlphabet{\mathsfit}{bold}{\encodingdefault}{\sfdefault}{bx}{n}

\newcommand{\R}{\mathbb{R}}

\DeclareMathOperator*{\argmin}{arg\,min}

\usepackage{amsmath}
\usepackage{mathrsfs}
\usepackage{amssymb}
\usepackage{mathtools}
\usepackage{soul}

\usepackage{algorithm}
\usepackage[noend]{algpseudocode}

\usepackage{bm}
\usepackage{enumitem}
\usepackage{graphicx}
\usepackage{subcaption}
\usepackage{hyperref}

\usepackage{parskip}
\usepackage[dvipsnames]{xcolor}

\usepackage{multirow}
\usepackage{booktabs}
\usepackage{tabularray}
\allowdisplaybreaks

\allowdisplaybreaks
\def\blue{\textcolor{blue}}

\newcommand{\Laura}[1]{\textbf{\textcolor{DarkOrchid}{#1}}}

\newtheorem{theorem}{Theorem}
\newtheorem{lemma}[theorem]{Lemma}

\theoremstyle{definition}

\theoremstyle{remark}
\newtheorem{remark}[theorem]{Remark}
\numberwithin{equation}{section}

\numberwithin{equation}{section}

\begin{document}
\title[Geodesic Distance Approximation] 
{A Primal-Dual Level Set Method for Computing Geodesic Distances}

\author{Hailiang Liu}
\address{$^\dagger$Department of Mathematics, Iowa State University, Ames, IA 50011}
\email{hliu@iastate.edu}

\author{Laura Zinnel}
\address{$^\ddagger$Department of Mathematics, Iowa State University, Ames, IA 50011}
\email{lezinnel@iastate.edu}

\thanks{ This work by Zinnel is partially supported by the National Science Foundation Grant DGE-2152117.}


\begin{abstract}
    The numerical computation of shortest paths or geodesics on surfaces, along with the associated geodesic distance, has a wide range of applications. Compared to Euclidean distance computation, these tasks are more complex due to the influence of surface geometry on the behavior of shortest paths.  {This paper introduces a primal-dual level set method for computing geodesic distances. A key insight is that the underlying surface can be implicitly represented as a zero level set, allowing us to formulate a constraint minimization problem. We employ the primal-dual methodology, along with   regularization and acceleration techniques, to develop our algorithm. This approach is robust, efficient, and easy to implement. We establish a convergence result for the high-resolution PDE system, and numerical evidence suggests that the method converges to a geodesic in the limit of refinement. }
\end{abstract}

\maketitle

\section{Introduction}
{A geodesic is a curve on a surface that minimizes local length, serving as a useful generalization of straight lines in Euclidean space.
Geodesic distances play a crucial role in both theoretical and applied problems, including those in computer vision \cite{Crane:2020}, conformal field theory \cite{Kastikainen:2022},  quantum field theory \cite{Chapman:2023}, and geophysical visualization \cite{Davis:2019}. Therefore, approximating the minimal geodesic path between two points on a surface is essential in many application domains.}

The approximation tasks can be framed as constrained optimization problems; however, they are often highly non-convex, and effective algorithms for computing geodesic distances are limited. For complex surfaces represented in 3D meshes, specialized algorithms such as the fast-marching method \cite{Kimmel:1998}, the Dijkstra algorithm (\cite{Davis:2019, Mitchell:1987, Memoli:2001}), or the heat flow method \cite{Crane:2013} are often employed to find the shortest path across the surface. Although effective, these methods often involve considerable complexity due to the construction of a mesh or graph representation. In contrast, our algorithm approximates geodesics on surfaces without requiring any surface discretization.

{ In this work, we adopt a different strategy: 
using the level set of a function to represent the surface, rather than employing  a triangulated mesh. In the literature, methods based on implicit level set surfaces are referred to as level set methods, which are particularly effective in image science applications, including image segmentation and rendering \cite{Osher:2003a}. 
In this work,  we introduce a primal-dual method for approximating geodesic distances and minimal geodesic paths on surfaces represented by level set functions. }

\subsection{Background}
In differential geometry \cite{Araujo:2024a}, the \textit{tangent space} to a surface $\Omega$ at a point $p$, is defined by 
\[T_p\Omega = \{\dot{\xi}(0)~|~\xi: (-\varepsilon, \varepsilon)\to \Omega \text{ is } C^{\infty} \text{ and } \xi(0) = p\}.
\]
 A surface is considered \textit{orientable} if there exists a continuous function $N$ on $\Omega$ such that for each $p$, the vector $N(p)$ is orthogonal to $T_p\Omega$. This field of normal vectors $N$ is referred to as an \textit{orientation} of $\Omega$. Notably, if $\Omega$ is the zero level set of function $\phi: \R^3\to \R$,  and $\gamma:[0,1]\to \Omega$ is a path on $\Omega$, then  
 $$
 \phi(\gamma(t))=0 \Rightarrow 
 \nabla \phi(\gamma(t))\cdot \dot \gamma(t)=0 \quad \text{for} \quad t\in [0, 1].   
 $$
 Thus, $\dot \gamma(t) \in T_p(\Omega)$. If the surface is smooth and the vector field is defined as 
\[N(p):=\frac{1}{|\nabla\phi(p)|}\nabla\phi(p)
\]
for every point $p=\gamma(t)$ along the curve, then $N$ provides an orientation of $\Omega$.
 
A curve $\gamma:[0,1]\to \Omega$ is a geodesic if and only if $\ddot{\gamma}$ is orthogonal to the tangent plane of $\Omega$ at $\gamma(t)$ for all $t\in[0,1]$ \cite{Pressley:2001}. 
Further, for a zero level set surface $\Omega$ with orientation $N$ as defined above, if $\ddot{\gamma}$ is parallel to $N(p)$ for every point $p  = \gamma(t)$ along the curve, 
then $\gamma$ is a geodesic.  

It is known that if $\gamma$ is a shortest path, then $\gamma$ is a geodesic; however, not every geodesic represents a shortest path \cite{Pressley:2001, Crane:2020}. {For example, if $p$ and $q$ are points on a sphere, then there is both a long and a short great circle arc joining these points. Both of these great circle arcs are geodesics, but the short great circle arc is the shortest geodesic \cite{Pressley:2001}.} The fact that not all geodesics are shortest paths, along with the fact that a shortest geodesic between two points is not necessarily unique, complicates the task of finding a minimal geodesic on a surface. 

\subsection{Problem Setup}\label{sec:prob-setup} 
Our goal is to develop an efficient algorithm for iteratively computing a geodesic and the geodesic distance (shortest path) between two points on a surface.  
Let $\Omega$ be a surface represented as the zero level set $\{x\in\R^3~|~ \phi(x)=0\}$ of a function $\phi:\R^3 \to \R$.  
For $p,q\in \Omega$ and a path $\gamma:[0,1]\to \R^3$ with $\gamma(0) = p$ and $\gamma(1) = q$, the \textbf{length} of $\gamma$ is calculated by 
\[
\int_0^1 |\dot \gamma(t)|~dt.
\]
We are particularly interested in paths from $p$ to $q$ on $\Omega$. That is, we seek path $\gamma$ to be in the set  $\Gamma(p,q,\Omega)$ defined by
\[
\Gamma(p,q,\Omega) = \left\{\gamma~ :~ \gamma(0) = p, \gamma(1) = q, \text{ and } \phi(\gamma(t)) = 0 ~\forall t\in [0,1] \right\}.
\]
Recall that a path $\gamma:[0,1]\to \R^3$ is a \textit{geodesic} if $\gamma \in \Gamma(p,q,\Omega)$ and $\ddot{\gamma}$ is orthogonal to the tangent plane of $\Omega$ at $\gamma(t)$ for all $t\in[0,1]$. Hence,
\begin{equation}\label{geodesic}
    \ddot{\gamma}\cdot \dot{\gamma} = 0.
\end{equation}
It has been shown that the shortest length path between any two points on a surface must be a geodesic \cite{Pressley:2001}.
The \textbf{geodesic distance} from $p$ to $q$ is the length of the \textit{shortest length path $\gamma$} from $p$ to $q$ on the surface $\Omega$. That is, the length of path $\gamma$ such that 
\begin{equation}\label{eq:min_geodesic}
\gamma = \argmin_{\xi \in \Gamma(p,q,\Omega)}\int_0^1|\dot \xi (t)|~dt. 
\end{equation}
Using this formulation, we propose determining the geodesic distance by solving the following optimization problem:
\begin{align*}
\min_{\gamma \in \Gamma(p, q, \Omega)}  \frac{1}{2}\int_0^1|\dot \gamma(t)|^2~dt, 
\end{align*}
which can be shown to be equivalent to (\ref{eq:min_geodesic}). 
To eliminate the level set constraint, a Lagrangian multiplier $\lambda(t) \in C^1[0, 1]$ is introduced in an inf-sup problem:
\[\inf_{\gamma}\sup_{\lambda} \left\{  \frac{1}{2}\int_0^1 |\dot{\gamma}(t)|^2~dt  + \int_0^1 \lambda(t)\phi(\gamma(t))~dt\right\}.
\]
A simple strategy is to alternate a gradient-like descent in $\gamma$ and a gradient-like ascent in $\lambda$. This yields the following iterative scheme:   
\[
\begin{cases}
 \lambda_{k+1}(t)  = \lambda_k(t)+\tau_\gamma \phi(\gamma_{k}(t)), \\
    \gamma_{k+1}(t)  = \gamma_k(t)-\tau_\lambda \left(-\ddot{\gamma}_k(t)+\lambda_{k+1}(t)\nabla\phi(\gamma_k(t))\right), 
\end{cases}
\]
where $\tau_\gamma, \tau_\lambda$ are tunable step sizes. However, such basic updates may not perform well in all cases due to stability issues inherent in the formulation.  To address this, we augment the functional with a regularization term of the form:
\[
\mathcal{L}_{\varepsilon}[\gamma,\lambda] = \frac{1}{2}\int_0^1 |\dot{\gamma}|^2~dt + \int_0^1\lambda(t)\phi(\gamma(t))dt -\frac{\varepsilon}{2}\int_0^1\lambda^2(t)~dt
\]
and enhance the gradient descent-ascent updates with a relaxation step, resulting in the following update scheme:  
\begin{equation*}\begin{cases}
    \lambda_{k+1}(t) = \frac{1}{1+\varepsilon\tau_{\lambda}}(\lambda_k(t)+\tau_{\lambda}\phi(\gamma_{k}(t))), \\
    \Tilde{\lambda}_{k+1}(t) = \lambda_{k+1}(t)+ \omega(\lambda_{k+1}(t)-\lambda_k(t)),\\
    \gamma_{k+1}(t) = \gamma_k(t) - \tau_{\gamma}(-\ddot{\gamma}_k(t)+\Tilde{\lambda}_{k+1}(t)\nabla\phi(\gamma_k(t))),
\end{cases}\end{equation*}
where $\varepsilon>0$ is a regularization parameter,  $\omega \in [0, 1]$ controls the degree of relaxation, and $\tau_\lambda, \tau_\gamma$ are tunable step sizes. This refinement has been shown to significantly improve  stability and numerical performance. Notably, the use of relaxation in this manner has also been explored in the Primal-Dual Hybrid Gradient (PDHG) method \cite{Chambolle:2011,Go2015} 
which addresses saddle-point problems of the primal-dual form: 
\[\min_{y\in Y}\max_{x\in X}  g(x) + \left(\langle Ax, y\rangle - f^*(y)\right).
\]
PDHG has proven to be an efficient algorithm for a wide range of problems, particularly in imaging applications,  and has gained popularity due to its strong convergence guarantees.
For further details, readers are referred to \cite{Chambolle:2011}, specifically Section \ref{sec:acceleration} for a brief description.

Our problem setup, which employs a novel level set representation, has not been previously explored. This approach establishes a primal-dual connection between the curve $\gamma$ and the multiplier $\lambda$,  resulting in an algorithm with potential applications in various fields.

\subsection{Contribution}
Our contribution is threefold.
\begin{itemize} 
\item In this paper, we introduce a simple and efficient algorithm for approximating the minimal geodesic path between two points on a level set surface.
\item We show that, under certain structural assumptions on the level set function, the continuous PDE system exhibits asymptotic convergence, thereby supporting the convergence of the discrete algorithm. Furthermore, we demonstrate that, upon convergence, the algorithm approximates a minimal geodesic curve. 
\item We also explored two variations of the base algorithm, both of which demonstrated promising performance in various examples. To illustrate, we present results for the base algorithm and its variations on examples such as the surfaces of a sphere, a torus, and the Stanford Bunny.
\end{itemize} 

\subsection{Related work} ${}$

{\bf Geodesic solvers.} 
Several popular methods exist for computing geodesic distances. The fast-marching method has been used to approximate the distances between points on a triangulated mesh representation of a surface \cite{Kimmel:1998}. The Dijkstra algorithm, an algorithm developed for finding the shortest paths in graphs, has been used to approximate the shortest paths on polyhedral surfaces \cite{Monneau:2010}. 
In addition, heat flow methods have been developed for the geodesic approximation using the heat kernel function \cite{Crane:2013}.
These methods typically require the surface to be represented as a triangulated mesh or as a graph. In contrast, our method can be applied directly to level set surfaces without requiring prior discretizaiton. For graph-based problems, we refer the reader to \cite{Chu:2017, Bungert:2023} for discussions on the use of Lipschitz embedding and Lipschitz learning to the study of  shortest paths on graphs. 

{\bf Level set methods.} Level set surface representations have proven to be valuable tools for various image processing tasks \cite{Osher:2003a}. Specifically, we are interested in using level set methods for brain imaging applications, such as sulci segmentation \cite{Torkaman:2017}, cortical surface reconstruction \cite{Cruz:2020a}, and cortical surface registration and labeling \cite{Joshi:2012}.

The application of level set methods involves initializing a level set function to represent a given surface. While this representation is not unique, a common choice is the signed distance function (SDF), where the value at each point is the distance to the nearest point on the surface, positive inside and negative outside.  
The SDF can be explicitly defined using a crossing-time approach or fast marching method, as described by \cite{Osher:2003a}. 
If the surface is given as a set of scattered points, a level set function can be reconstructed by iteratively refining the level set function to minimize an energy functional based on the data's proximity and smoothness \cite{Zhao:2001}. Alternatively, deep learning methods have also been implemented for generating level set surface representations directly from input data, such as point clouds or image features \cite{Park:2019}. Refer to Section \ref{sec:examples} for an explanation of how the SDF satisfies the structural assumption outlined in Section 3.3 for the example shown in Figure \ref{fig:bunny}. 

Level-set methods have been applied to model the motion of curves on surfaces. In one approach introduced by \cite{Cheng:2002}, a level set function represents the surface, while an additional  level set function is
used to define the curve.  The curve is then represented by the intersection of the two zero level sets of these two functions. 
Our approach represents the curve $\gamma: [0,1]\to \R^3$ directly as a parametrized curve, rather than defining $\gamma$ through the intersection of two zero level sets.

{\bf Primal-dual algorithms.} Primal-dual formulations for constrained optimization have been extensively studied,  with applications ranging from imaging to broader convex optimization problems (see \cite{Rockafellar:1970, Goldstein:2009, Esser:2010, Chambolle:2011, Boyd:2010}). These methods for managing constraints have played a significant role in the development of primal-dual algorithms. Building on this foundation, our work incorporates a level-set constraint into the framework.

{\bf Acceleration.} While gradient descent/ascent is a reliable first-order optimization method, it often suffers from slow convergence in many practical problems. This limitation  has motivated  the development of accelerated variants that incorporate momentum, such as  Polyak's heavy ball method (1964) and Nesterov’s accelerated gradient descent (1983). In the context of calculus of variations problems defined over functions, a PDE-based acceleration framework was introduced by Benyamin, Calder, Sundaramoorthi and Yezzi \cite{BCS2020}, further analyzed in \cite{CY2019}. These works show improved efficiency via wave-equation dynamics and explicit time-stepping schemes. In this work, we adopt the relaxation method, common used in the Primal-Dual Hybrid Gradient (PDHG) algorithm \cite{Chambolle:2011,Go2015}.

\subsection{Organization of the Paper} 
In Section 2, we define the geodesic approximation problem. Section 3 introduces an algorithm to solve it, followed by an analysis of the algorithm's behavior in Section 4. Finally, concluding remarks and discussions are provided at the end.

\section{Method}\label{sec:method}

The uniqueness of geodesics in three-dimensional spaces depends on the specific geometry and topology of the space -- particularly aspects such as curvature, singularities, and the presence of antipodal points or a cut locus. For most pairs of points on a given surface, the geodesic connecting  them is unique. Therefore, we will focus on the case where a unique geodesic exists and defer discussion of non-unique geodesics to a later section. 

\subsection{Optimization Problem}
From now on, we define 
$$
L(p, q)=\int_0^1 |\dot \gamma(t)|dt
$$
as the \textbf{geodesic length} from $p$ to $q$, provided $\gamma(t)$ is a geodesic curve which is a solution to the minimization problem (\ref{eq:min_geodesic}).

We want to find the geodesic distance $L(p,q)$ from $p$ to $q$ on surface $\Omega$, represented by $\Omega = \{x~|~ \phi(x) = 0\}$ for some function $\phi: \R^3 \to \R$. Define
\[d^2(p,q) = \min_{\gamma \in \Gamma(p, q, \Omega)}\int_0^1 |\dot{\gamma}(t)|^2~dt.
\]
By Cauchy-Schwarz, 
\[
|p-q| \leq \int_0^1 |\dot{\gamma}(t)|~dt \leq \left(\int_0^1 1~dt\right)^{1/2}\left(\int_0^1 |\dot{\gamma}(t)|^2~dt\right)^{1/2} = \left( \int_0^1 |\dot{\gamma}(t)|^2~dt\right)^{1/2},\] 
hence we have 
$$
|p-q|\leq L(p, q)\leq d(p, q). 
$$
Note that for a geodesic, $L(p, q)=d(p, q)$.  More precisely, we have the following result. 
\begin{lemma}
  If $\gamma$ is a geodesic on $\Omega$ from $p$ to $q$, then $L(p,q) = d(p,q)$.
 Moreover, the following statements are equivalent:\\
(i) $\gamma$ is a constant-speed geodesic. \\ 
(ii) $\gamma \in C^1[0, 1]$ and $|\dot \gamma(t)| = d(\gamma(0), \gamma(1))$ for all $t \in[0, 1]$. \\
(iii) Let $m>1$, $\gamma$ solves a class of minimization problems 
$$
\min_{\gamma \in \Gamma(p, q, \Omega)} \left\{ 
\int_0^1 |\dot \gamma(t)|^m dt\right\}.
$$
\end{lemma}
\begin{proof}
(i) $\Leftrightarrow$ (ii): First, recall that by the definition of a geodesic, the path $\gamma$ minimizes arc length, i.e.,
$$
L(p, q)=\int_0^1 |\dot \gamma(t)|dt =d(p, q).
$$
Since $\gamma$ is a geodesic, it satisfies the Euler–Lagrange equations associated with arc length minimization.
In particular, we compute:
    \[\frac{d}{dt}|\dot{\gamma}(t)| = \frac{1}{|\dot{\gamma}(t)|}(\dot{\gamma}(t)\cdot \ddot{\gamma}(t)) = 0.\]
    Therefore, $|\dot{\gamma}|$ is constant in $t$.
    Therefore, $\gamma$ moves at constant speed. Thus,  
\begin{align*}
    L(p,q)  = \int_0^1|\dot{\gamma}(t)|~dt = |\dot{\gamma}(t)|
    = \left(\int_0^1 |\dot{\gamma}(t)|^2~dt\right)^{1/2}
    = d(p,q).
\end{align*}
The equivalent between (i) and (ii) follows directly (see also a proof in \cite{Pressley:2001} on page 216).\\
 (ii) $\Rightarrow$ (iii): Assume $\gamma \in C^1[0, 1]$ and $|\dot \gamma(t)| = d(\gamma(0), \gamma(1))$ for all $t \in[0, 1]$. For $m>1$, choose $m^*$ such that $\frac{1}{m}+\frac{1}{m^*} = 1$, then H\"{o}lder's inequality gives 
\[\int_0^1 |\dot{\gamma}(t)|~dt \leq \left(\int_0^1 1^{m^*}~dt\right)^{\frac{1}{m^*}}\left(\int_0^1 |\dot{\gamma}(t)|^{m}~dt\right)^{\frac{1}{m}} = \left(\int_0^1 |\dot{\gamma}(t)|^{m}~dt\right)^{\frac{1}{m}}.\]
This implies 
$$
\left(\int_0^1 |\dot{\gamma}(t)|~dt\right)^m \leq \int_0^1 |\dot{\gamma}(t)|^m~dt.
$$
Since  $|\dot{\gamma}(t)| = d(p,q)$ is a constant, the inequality becomes an equality 
\[
\int_0^1 |\dot{\gamma}(t)|^m~dt = d(p,q)^m = \left(\int_0^1 |\dot{\gamma}(t)|~dt\right)^m.
\]
Therefore, $\gamma$ achieves the minimum of the functional 
$\int_0^1 |\dot \gamma(t)|^m dt$ over $\Gamma(p, q, \Omega)$, proving (iii).  

(iii) $\Rightarrow$ (i):  Assume $\gamma$ minimizes  
$\int_0^1 |\dot \gamma(t)|^m dt$. This means H\"{o}lder's inequality must be tight as an equality, which only occurs  when $|1|^{m^*} = 1$ and $|\dot{\gamma}^*(t)|^m$ are linearly dependent. This implies $|\dot\gamma(t)|$ must be constant almost everywhere. Therefore, $\gamma$ is a constant-speed geodesic. 
\end{proof}
Considering the smoothness of the underlying functional defined by $d(p, q)$, we propose to find the geodesic distance from $p$ to $q$ by solving the following optimization problem:
\begin{align*}
\min_{\gamma \in \Gamma(p, q, \Omega)}  \frac{1}{2}\int_0^1|\dot \gamma(t)|^2~dt.    
\end{align*}
To eliminate the level-set constraint, we introduce a Lagrangian multiplier $\lambda(t) \in C^1[0, 1]$, so that the constraint optimization problem can be reformulated as  an inf-sup problem:
\[\inf_{\gamma}\sup_{\lambda} \mathcal{L}[\gamma,\lambda],\]
where  $\mathcal{L}$ is defined by 
\[
\mathcal{L}[\gamma,\lambda] = \frac{1}{2}\int_0^1 |\dot{\gamma}(t)|^2~dt  + \int_0^1 \lambda(t)\phi(\gamma(t))~dt.
\]

\subsection{Optimality  conditions}

At the optimal pair $(\gamma,\lambda)$, the variational derivative $\delta \mathcal{L} = 0$. Observe that 
\begin{align*}
    \delta \mathcal{L} & = \int_0^1 \dot{\gamma}(t)\delta\dot{\gamma}(t)~dt + \int_0^1\lambda(t) \nabla\phi(\gamma)\cdot \delta\gamma ~dt + \int_0^1 \phi(\gamma(t))\delta\lambda~dt\\
    & = \dot{\gamma}\delta\gamma\big{\rvert}_0^1 - \int_0^1 \ddot{\gamma}\delta\gamma~dt+\int_0^1\lambda(t) \nabla \phi(\gamma) \cdot \delta\gamma~dt +  \int_0^1 \phi(\gamma(t)) \delta\lambda~dt\\
    & = \dot{\gamma}\delta\gamma\big{\rvert}_0^1+\int_0^1(-\ddot{\gamma}+\lambda(t)\nabla\phi(\gamma))\cdot \delta\gamma~dt +  \int_0^1 \phi(\gamma(t))\delta\lambda~dt.
\end{align*}
Note that $\gamma(0)=p$ and $\gamma(1)=q$ are fixed, which implies that $\delta \gamma=0$ at $t=0$ and $t=1$. Consequently, the optimal solution must satisfy the following conditions: 
\[\begin{cases}
    -\ddot{\gamma}(t)+\lambda(t)\nabla\phi(\gamma(t)) = 0,\\
    \phi(\gamma(t)) = 0,\quad  \forall t\in [0,1],\\
    \gamma(0) = p,\quad 
    \gamma(1) = q.
\end{cases}\]

According to \cite[Definition 9.1.1]{Pressley:2001},  a curve $\gamma$ on a surface $\mathcal{S}$ is called a \textit{geodesic} if its acceleration vector $\ddot{\gamma}(t)$ is either zero or perpendicular to the tangent plane of the surface at the point $\gamma(t)$; that is, $\ddot{\gamma}(t)$ is parallel to the unit normal at that point, for all values of the parameter $t$. We now state the following result. 
\begin{lemma}
   A point ($\lambda$,$\gamma$) satisfies the optimality conditions if and only if $\gamma$ is a geodesic.
\end{lemma}
\begin{proof}
($\Rightarrow$) Suppose ($\lambda$,$\gamma$) satisfies the optimality conditions.
Then, by definition,  $\phi(\gamma(t)) = 0$ for all $t\in [0,1]$,  $\gamma(t)$ lies on the surface $\phi=0$. Further, the conditions $\gamma(0) = p$ and $\gamma(1) = q$ imply $\gamma$ is a path from $p$ to $q$ on the surface. The last condition says $\ddot{\gamma}(t)=\lambda(t)\nabla\phi(\gamma(t))$ for all $t\in(0, 1)$, that is, the acceleration vector $\ddot{\gamma}$ is parallel to $\nabla\phi(\gamma(t))$ for all $t\in (0,1)$. Since $\nabla\phi(\gamma(t))$ is normal to the surface at  $\gamma(t)$, $\ddot{\gamma}(t)$ is parallel to its unit normal for all $t\in (0,1)$.
Hence, $\gamma$ has zero tangential acceleration and lies entirely on the surface, satisfying the definition of a geodesic.

($\Leftarrow$) Now assume $\gamma$ is a geodesic on the surface defined by $\phi=0$, connecting points $p$ and $q$. 
Then, $\phi(\gamma(t)) = 0$ for all $t\in (0, 1)$, $\gamma(0)= p$, and $\gamma(1) = q$ must be satisfied. Since $\gamma$ is a geodesic, its acceleration vector  $\ddot{\gamma}(t)$ is orthogonal to the surface at each point $\gamma(t)$, i.e., normal to the tangent plane. Because $\nabla\phi(\gamma(t))$ is normal to the tangent plane at $\gamma(t)$, there exists a scalar function $\lambda(t)$ such that  $\ddot{\gamma}(t) = \lambda(t)\nabla\phi(\gamma(t))$ for all $t\in (0,1)$.Therefore, the pair $(\lambda, \gamma)$ satisfies the optimality conditions. 
\end{proof}

\subsection{Gradient Descent-Ascent}\label{sec:grad_desc_asc}
For the inf-sup problem, we can use a Gradient Descent-Ascent approach to update $\gamma(t)$ and $\lambda(t)$ for all $t\in [0,1]$. In the $(k+1)th$ iteration, the updates are defined by
\begin{align*}
\lambda_{k+1}(t) & = \lambda_k(t)+\tau_{\lambda}\left(\frac{\delta}{\delta\lambda}\mathcal{L}[\gamma,\lambda]\right),\\
    \gamma_{k+1}(t) & = \gamma_k(t)-\tau_{\gamma}\left(\frac{\delta}{\delta\gamma}\mathcal{L}[\gamma,\lambda]\right),
\end{align*}
where $\tau_{\gamma}$ and $\tau_{\lambda}$ are tunable step sizes. Note that 
\begin{align*}
    \frac{\delta}{\delta\lambda} \mathcal{L}[\gamma,\lambda] & = \phi(\gamma(t)), \\
 \frac{\delta}{\delta\gamma}\mathcal{L}[\gamma,\lambda] & = -\ddot{\gamma}(t)+\lambda(t)\nabla\phi(\gamma(t)). 
\end{align*}
This results in the following updates:
\begin{equation}\begin{cases}
 \lambda_{k+1}(t)  = \lambda_k(t)+\tau_{\lambda}\phi(\gamma_{k}(t)), \\
    \gamma_{k+1}(t)  = \gamma_k(t)-\tau_{\gamma}\left(-\ddot{\gamma}_k(t)+\lambda_{k+1}(t)\nabla\phi(\gamma_k(t))\right).
\end{cases}\label{eqs:GDA}\end{equation}

One issue with the algorithm is that the maximization of  $\mathcal{L}$ with respect to $\lambda$ may not be well-defined, potentially causing a lack of convergence for  $\lambda_k(t)$. When this update scheme was applied to approximate the geodesic distance between antipodal points on the surface of a unit sphere, the algorithm failed to  converge. Figure \ref{fig:GDA-error} illustrates the absolute error in the approximate geodesic distance relative to  the number of iterations. 
\begin{figure}[H]
    \centering
    \includegraphics[width=0.5\linewidth]{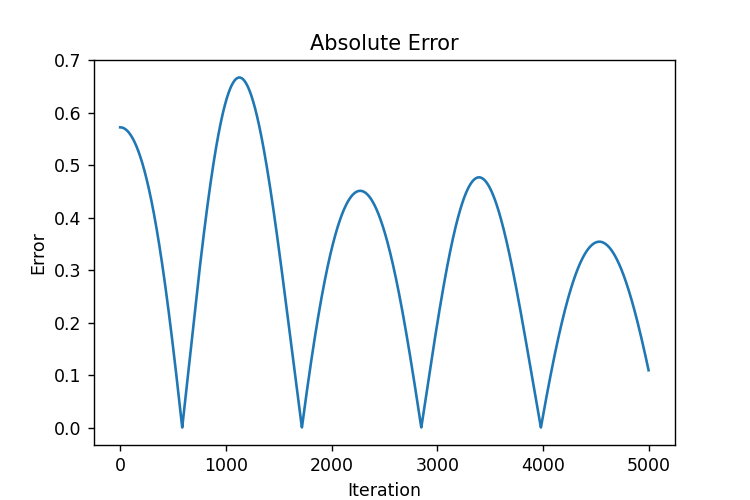}
    \caption{Absolute error of the approximated geodesic distance between antipodal points on the sphere, plotted against the number of iterations for the update scheme described above.}
    \label{fig:GDA-error}
\end{figure}

\subsection{Regularization}
To enhance the algorithm's performance,  we  introduce a regularization term to the functional  $\mathcal{L}$ to ensure a well-defined maximum of $\mathcal{L}$ with respect to $\lambda$. For $\varepsilon>0$, we define
\[\mathcal{L}_{\varepsilon}[\gamma,\lambda] = \frac{1}{2}\int_0^1 |\dot{\gamma}|^2~dt + \int_0^1\lambda(t)\phi(\gamma(t))dt -\frac{\varepsilon}{2}\int_0^1\lambda^2(t)~dt.\]
We then have  
\begin{align*}
  \frac{\delta}{\delta\lambda} \mathcal{L}_{\varepsilon}[\gamma,\lambda] & = \phi(\gamma(t)) - \varepsilon\lambda(t), \\
    \frac{\delta}{\delta\gamma}\mathcal{L}_{\varepsilon}[\gamma,\lambda] & = -\ddot{\gamma}(t)+\lambda(t)\nabla\phi(\gamma(t)).
\end{align*}
Thus, the Gradient Descent-Ascent updates become
\[\begin{cases}
   \lambda_{k+1}(t) = \lambda_k(t) + \tau_{\lambda}(\phi(\gamma_{k}(t))-\varepsilon\lambda_{k+1}(t)),\\
    \gamma_{k+1}(t) = \gamma_k(t)-\tau_{\gamma}(-\ddot{\gamma}_k(t)+\lambda_{k+1}(t)\cdot \nabla\phi(\gamma_k(t))).
\end{cases}\]
This can be rewritten as
\[\begin{cases}
    \lambda_{k+1}(t) = \frac{1}{1+\varepsilon\tau_{\lambda}}(\lambda_k(t)+\tau_{\lambda}\phi(\gamma_{k}(t))),\\
    \gamma_{k+1}(t) = \gamma_k(t)-\tau_{\gamma}(-\ddot{\gamma}_k(t)+\lambda_{k+1}(t)\cdot \nabla\phi(\gamma_k(t))).\\
\end{cases}\]

\subsection{Acceleration}\label{sec:acceleration}
The primal-dual hybrid gradient (PDHG) algorithm is an efficient first-order optimization algorithm for solving saddle point problems \cite{Chambolle:2011}. Generally,  if $\langle \cdot, \cdot \rangle$ represents  an inner product on Hilbert spaces $X$ and $Y$,  with the norm defined as $\|\cdot \| = \langle \cdot, \cdot \rangle ^{1/2}$, then the PDHG algorithm is employed to solve saddle-point problems of the primal-dual form 
\[\min_{y\in Y}\max_{x\in X}  g(x) + \left(\langle Ax, y\rangle - f^*(y)\right),
\]
where functionals $g:X\to [0,+\infty]$ and $f:Y\to [0,+\infty]$ are proper, convex, and lower semi-continuous, $A:X\to Y$ is a continuous linear operator, and $f^*$ is the convex conjugate of $f$.
The PDHG algorithm uses gradient ascent and descent steps as follows:
\[\begin{cases}
    x_{k+1} = \text{prox}_{\tau_x g}(x_{k} - \tau_x L^*y_{k}), \\
    \Tilde{x}_{k+1} = x_{k+1}+ \omega(x_{k+1} -x_k),\\
    y_{k+1} = \text{prox}_{\tau_y f^*}(y_{k} - \tau_y L\Tilde{x}_{k+1}),
\end{cases}\]
where $\tau_y,\tau_x$, and $\omega$ are constants. 
The convergence of this algorithm is guaranteed when $\tau_x\tau_y\|A\|^2 <1$. In \cite{Chambolle:2011}
, the convergence of the PDHG algorithm was analyzed for $\omega = 1$, demonstrating a convergence rate of $O(1/k)$ in finite dimensional spaces.

The relaxation step in the Primal-Dual Hybrid Gradient (PDHG) algorithm is designed to accelerate convergence by leveraging information from previous iterations to improve the current update -- a concept similar to momentum in optimization. Inspired by such relaxation step, we introduce the following update: 
\begin{equation}\label{eqs:PDHG_update}
\begin{cases}
    \lambda_{k+1}(t) = \frac{1}{1+\varepsilon\tau_{\lambda}}(\lambda_k(t)+\tau_{\lambda}\phi(\gamma_{k}(t))), \\
    \Tilde{\lambda}_{k+1}(t) = \lambda_{k+1}(t)+ \omega(\lambda_{k+1}(t)-\lambda_k(t)),\\
    \gamma_{k+1}(t) = \gamma_k(t) - \tau_{\gamma}(-\ddot{\gamma}_k(t)+\Tilde{\lambda}_{k+1}(t)\nabla\phi(\gamma_k(t))).
\end{cases}\end{equation}
where $\omega \in (0, 1]$ is the relaxation parameter. This creates a sort of ``extrapolated" variable $\tilde \lambda$, which combines the current and previous values of the $\lambda$ variable. Before implementing the algorithm, we conduct a primary analysis using ODE approach to gain deeper insights into the behavior and stability of the update.

Note that if $\phi$ is linear --for instance,  $\phi(\gamma)=a\cdot \gamma$, where $a\in \mathbb{R}^3$ is a constant vector,  then the objective functional falls into the class suitable for applying  the PDHG method \cite{Chambolle:2011}. In the present infinite dimensional setting, using the optimal relaxation parameter $\omega=1$, the update becomes 
\begin{equation}\label{eqs:PDHG_update+}
\begin{cases}
    \lambda_{k+1}(t) = \frac{1}{1+\varepsilon\tau_{\lambda}}(\lambda_k(t)+\tau_{\lambda}\phi(\gamma_{k}(t))), \\
    \gamma_{k+1}(t) = \gamma_k(t) - \tau_{\gamma}(-\ddot{\gamma}_{k+1}(t)+(2\lambda_{k+1}(t) -\lambda_k(t)) \nabla\phi(\gamma_k(t))).
\end{cases}\end{equation}
A detailed convergence rate analysis will be presented in Section \ref{sec5}: the result asserts that a convergence rate of $O(1/k)$ is ensured if $\tau_\lambda \tau_\gamma |a|^2 <1$. This is consistent with the known result for PDHG. 

\section{Analysis of dynamic behavior}\label{sec:analysis}
One effective way for analyzing optimization algorithms is to first examine their continuous limit through differential equations. This method provides valuable insights into the convergence properties of the algorithm. 

We begin by reformulating the proposed algorithm as a PDE system. 

\subsection{PDE System}\label{sec:PDEsystem} The discrete dynamics can be viewed as a discretization of the continuous PDE system. To illustrate this, we note that 
    \begin{align*}
        \Tilde{\lambda}_{k+1} & = (1+w)\lambda_{k+1}-w\lambda_k\\
        & = \frac{1+w}{1+\varepsilon\tau_{\lambda}}(\lambda_k+\tau_{\lambda}\phi(\gamma_k))-w\lambda_k\\
        & = \frac{1-\varepsilon w\tau_{\lambda}}{1+\varepsilon\tau_{\lambda}}\lambda_k+\frac{(1+w)\tau_{\lambda}}{1+\varepsilon\tau_{\lambda}}\phi(\gamma_k)\\
        & = \left(1-\frac{1+w}{1+\varepsilon \tau_{\lambda}}\varepsilon\tau_{\lambda}\right)\lambda_k+\frac{(1+w)\tau_{\lambda}}{1+\varepsilon\tau_{\lambda}}\phi(\gamma_k).
    \end{align*}
    Hence,
    \begin{align*}
        \frac{\lambda_{k+1}-\lambda_k}{\tau_{\lambda}} & = -\frac{\varepsilon}{1+\varepsilon\tau_{\lambda}}\lambda_k + \frac{1}{1+\varepsilon\tau_{\lambda}}\phi(\gamma_k),\\
        \frac{\gamma_{k+1}-\gamma_k}{\tau_{\gamma}} & = \ddot{\gamma}_k(t)-\nabla\phi(\gamma_k)\cdot\left[ \left(1-\frac{1+w}{1+\varepsilon \tau_{\lambda}}\varepsilon\tau_{\lambda}\right)\lambda_k+\frac{(1+w)\tau_{\lambda}}{1+\varepsilon\tau_{\lambda}}\phi(\gamma_k) \right].
    \end{align*}
    For any $k>0$, let $\tau: =\tau_{\gamma}=\frac{\tau_{\lambda}}{1+\varepsilon\tau_{\lambda}}$ and  $s_k=k\tau$. We assume that $(\lambda_k, \gamma_k)$ correspond to values at $s=s_k$ for some sufficiently smooth functions $(\lambda(s, t), \gamma(s, t))$.
 As $\tau  \to 0$, while keeping $\alpha=(1+\omega)\tau$, we have 
   \begin{equation}\label{eq:equi0}
   \begin{cases}
        \lambda' & = -\varepsilon\lambda + \phi(\gamma),\\
        \gamma' & = \ddot{\gamma} - \nabla\phi(\gamma)\cdot((1-\alpha\varepsilon)\lambda + \alpha\phi(\lambda)),
    \end{cases}
    \end{equation}
    where $\lambda = \lambda(t, s)$ and $\lambda' = \frac{\partial}{\partial s}\lambda(t,s)$ for the training step index $s$ and curve parameter $t$. 

Regarding the PDE system, we state the following result:  
\begin{theorem} As $\tau_\lambda, \tau_\gamma \to 0$, the update in \ref{eqs:PDHG_update} can be expressed as a discrete-time update of the following PDE system:
\begin{align*}
        \lambda' & = -\varepsilon\lambda + \phi(\gamma),\\
        \gamma' & = \ddot{\gamma} - \nabla\phi(\gamma)\cdot((1-\alpha\varepsilon)\lambda + \alpha\phi(\gamma))
    \end{align*}
    with the initial condition $\lambda(t, 0)=\lambda_0(t)$ and $\gamma(t, 0)=\gamma_0(t)$. If $\nabla \phi$ is Lipschitz continuous, then there exists a unique solution for the PDE system. An equilibrium of the PDE system satisfies: 
  \begin{equation}\label{eq:equil}
      \begin{cases}
        -\varepsilon\lambda^* +\phi(\gamma^*) = 0, \\
        \ddot{\gamma}^* - ((1-\alpha\varepsilon)\lambda^*+\alpha\phi(\gamma^*))\nabla\phi(\gamma^*) = 0.
        \end{cases}
    \end{equation}
\end{theorem}
The goal is to show that, as $s\to +\infty$, the PDE system converges to the optimal solution $\lambda^*(t)$ and $\gamma^*(t)$ governed by (\ref{eq:equil}).  Additionally, as $\epsilon \to 0$, the curve  $\gamma(t)$ approaches a geodesic.  

\subsection{Equilibrium curves}

\begin{lemma}\label{lem:approx_geo}
    If $\gamma^*, \lambda^*$ is an equilibrium solution to the system, 
    \[\begin{cases}
        \lambda' & = -\varepsilon\lambda + \phi(\gamma),\\
        \gamma' & = \ddot{\gamma} - ((1-\alpha\varepsilon)\lambda + \alpha\phi(\gamma))\nabla\phi(\gamma),
    \end{cases}\]
    then $\gamma^*$ approximates a geodesic on $\Omega$ when $\varepsilon$ is small. Specifically, as $\varepsilon \to 0$, both $\phi(\gamma(t)) \to 0$ and $\ddot{\gamma}^*\cdot \dot{\gamma}(t) \to 0$. 
\end{lemma}
\begin{proof} 
    Assume $\gamma^*$ and $\lambda^*$ is an equilibrium solution to the above system.
    Equations \ref{eq:equil} imply that 
    \[\varepsilon \lambda^* = \phi(\gamma^*).\]
    That is, $\phi(\gamma^*) \to 0$, as $\varepsilon\to 0$ , which means that $\gamma^*\to \Gamma(p,q,\Omega)$ for all $t\in[0,1]$.
    Also, taking the derivative gives us 
    \begin{equation}
    \varepsilon \dot{\lambda}^* = \dot{\gamma}^*\cdot\nabla\phi(\gamma^*).
    \end{equation}
     {Therefore, $\dot{\gamma}^*\cdot\nabla\phi(\gamma^*)\to 0$ as $\varepsilon \to 0$, so $\nabla\phi(\gamma^*)$ approximates the normal direction to the surface.}

    Further, equations \ref{eq:equil} also imply $\ddot{\gamma}^* = ((1-\alpha\varepsilon)\lambda^* + \alpha\phi(\gamma^*))\nabla\phi(\gamma^*)$.
    {Hence, $\ddot{\gamma}^*$ is parallel to $\nabla\phi(\gamma^*)$, an approximation of the normal direction to the surface.} 
    Thus, according to the definition \ref{geodesic} of a geodesic, $\gamma^*$ approximates a geodesic when $\varepsilon$ is small.  
\end{proof}

It can also be shown that $\gamma$ is a good approximation to the shortest geodesic from $p$ to $q$ on $\Omega$.

\begin{lemma}\label{lem:min_geo}
    If $\gamma^*, \lambda^*$ is an equilibrium solution to the system
   \[\begin{cases}
        \lambda' & = -\varepsilon\lambda + \phi(\gamma),\\
        \gamma' & = \ddot{\gamma} - \nabla\phi(\gamma)((1-\alpha\varepsilon)\lambda + \alpha\phi(\gamma)),
    \end{cases}\]
    then, when $\varepsilon$ is small, $\gamma^*$ approximates a minimal geodesic for the surface represented by $\{x|\quad  \phi(x)=0\}$. 
    Specifically, $\gamma^*$ minimizes $\int_0^1|\dot{\gamma}^*(t)|^2~dt$ under the constraint $\phi(\gamma(t)) = \varepsilon\lambda^*(t)$ (and as $\varepsilon \to 0$, both $\phi(\gamma(t)) \to 0$ and $\ddot{\gamma}^*\cdot \dot{\gamma}(t) \to 0$.)
\end{lemma}

\begin{proof}It suffices to show that $\gamma^*(t)$ is the solution to the constrained minimization:
$$
\\min \int_0^1 |\dot \gamma(t)|^2dt \quad s.t. \quad \phi(\gamma(t))=\epsilon \lambda^*(t), \quad t\in [0, 1]. 
$$
To do so, we 
let $\delta \in \mathbb{R}$ be a small parameter, and define 
\[
J(\delta) = \left\{ \int_0^1|\dot \gamma^*(t)+\delta \dot a(t)|^2~dt, \quad \phi(\gamma^*(t)+\delta a(t))=\varepsilon\lambda^*(t) \quad \forall t \in [0, 1]\right\},
\]
where $a\in C^1[0, 1]$ is an arbitrary perturbation that satisfies $a(0)=a(1)=0$. 
Observe that 
\begin{align*}
   \frac{d}{d\delta} \int_0^1|\dot{\gamma}^*(t)+\delta\dot{a}(t)|^2~dt & 
    = \int_0^1 2( \dot \gamma^*(t) + \delta \dot{a}(t))\cdot\dot{a}(t)~dt. 
   \end{align*}
Thus, 
\begin{align*}
    \lim_{\delta\to 0} \frac{d}{d\delta} \int_0^1|\dot{\gamma}^*(t)+\delta\dot{a}(t)|^2~dt 
    & = \int_0^1 2\dot{a}(t)\cdot \dot{\gamma}^*(t)~dt\\
    & =-2 \int_0^1 a(t)\cdot \ddot \gamma^*(t)dt.
\end{align*}
This relies on the fact that $a(0)=a(1)=0$.  On the other hand, the constraint 
$\phi(\gamma^*(t)+\delta a(t))=\varepsilon \lambda^*(t)$ for any $\delta$ yields 
$$
\nabla \phi(\gamma^*(t)+\delta a(t))\cdot a(t)=0.
$$
Since $\ddot{\gamma}^* = ((1-\alpha\varepsilon)\lambda^*+\alpha\phi(\gamma^*))\nabla\phi(\gamma^*(t))$, it follows that  
\[\ddot{\gamma}^*(t) \cdot a(t) = 0.
\]
Consequently, we have  
\begin{align*}
    \lim_{\delta\to 0} \frac{d}{d\delta} \int_0^1|\dot{\gamma}^*(t)+\delta\dot{a}(t)|^2~dt & 
       =-2 \int_0^1 a(t)\cdot \ddot \gamma^*(t)dt= 0.
\end{align*}
 Therefore, $\delta=0$ is a critical point for $J(\delta)$.
Now, taking the second derivative, 
\begin{align*}
    \frac{d^2}{d\delta^2} \int_0^1|\dot{\gamma}^*(t)+\delta\dot{a}(t)|^2~dt 
     = \int_0^1 2|\dot{a}(t)|^2 dt \geq 0.
\end{align*}
Therefore, $J(\delta)$ is convex, so $\delta=0$ is a global minimizer of $J(\delta)$.  
Thus, $\gamma^*$ minimizes
$\int_0^1|\dot{\gamma}^*(t)|^2~dt$ under the constraint $\phi(\gamma(t)) = \varepsilon\lambda^*(t)$, and when $\varepsilon$ is small, $\gamma^*$ approximates one of the shortest geodesics from $p$ to $q$. 
\end{proof}

\subsection{Convergence analysis} In this section, we present the main theoretical result of this work,  providing the convergence analysis in continuous-time PDE systems. Let $U^*=(\lambda^*(t), \gamma^*(t))$ be the steady solution, then 
\begin{equation} \label{ee} 
\begin{cases}
        -\varepsilon\lambda^*(t) +\phi(\gamma^*(t)) = 0, \\
        \ddot{\gamma}^*(t)-[(1-\alpha\varepsilon)\lambda^*(t)+\alpha\phi(\gamma^*(t))]\nabla \phi(\gamma^*(t)) = 0.
    \end{cases}
\end{equation} 
It is reasonable to consider the domain in which  $\phi$ is suitably small. More precisely, we make the following assumption. 

{\textbf{Assumption A.} There exists a constant $\nu>0$ and a constant $a>0$ such that  
\[ |\nabla\phi(x)|^2 \geq \nu  \quad \text{and} \quad 2 a \|D^2 \phi(x)\| \leq \nu
\]
for $x\in \Sigma$, where 
$$
\Sigma = \{x~|~ |\phi(x)|\leq a\}. 
$$
}

\begin{theorem} \label{thm5} Consider the PDE system 
\[\begin{cases}
    \lambda'(t,s)  = -\varepsilon\lambda(t,s) +\phi(\gamma(t,s)),\\
    \gamma'(t,s) = \ddot{\gamma}(t,s)-[(1-\alpha\varepsilon)\lambda(t,s)+\alpha\phi(\gamma(t,s))]\nabla \phi(\gamma(t,s)), 
\end{cases}\]
with initial condition $\gamma_0(t)$ satisfying $\phi(\gamma_0(t))=0$
for $t=0$ and $t=1$. Suppose Assumption A holds. Then, for appropriate choices of $\epsilon$ and $\alpha$, satisfying  $1/2< \alpha \epsilon <1$, the system converges exponentially to the equilibrium curve governed by (\ref{ee}).    
\end{theorem}
The proof presented below comprises two steps:

{\bf Step 1}: We introduce a Lyapunov function defined for  $U = (\lambda, \gamma)$ as 
$$
J(U):=\int_0^1 (-\varepsilon\lambda + \phi(\gamma(t, s))^2~dt + \mu\int_0^1|\ddot{\gamma}(t,s)-[(1-\alpha\varepsilon)\lambda(t,s)+\alpha\phi(\gamma(t,s))]\nabla \phi(\gamma(t,s))|^2~dt,
$$
where the parameter $\mu$ is to be determined so that $J$ along the PDE dynamics is non-increasing in $s$, at least for large $s$. Thus,  all solutions will lie strictly in the sub-level set defined by  
$$
Z=\{U, \quad J(U)\leq J(U_0)<\infty\}. 
$$
{\bf Step 2:} We fix $\mu$ and show that there is a $\beta>0$ such that the decay rate of $J$ is 
$$
\frac{d}{ds}J \leq -\beta J. 
$$
After these two essential steps,  we will be able to prove the time-asymptotic convergence result expressed as 
$$
J(U) \sim  O(1) e^{-\beta s} 
$$
as $s$ becomes large. 

\begin{proof} When evaluating $J$ along the PDE system, it can be rewritten as 
\begin{align*}
    J & = \int_0^1\lambda'^2~dt+\mu\int_0^1|\gamma'|^2~dt, 
\end{align*}
with $\mu>0$ to be determined. Taking the time-derivative of $J$, we have 
\begin{align*}
     \frac{d}{ds} J & = 2 \int_0^1\lambda'\lambda''~dt + 2 \mu\int_0^1 \gamma'\cdot  \gamma''~dt\\
    & = 2 \int_0^1\lambda'(-\varepsilon\lambda'+\nabla\phi(\gamma)\cdot\gamma')~dt \\ 
    &+ 2 \mu\int_0^1\gamma' \cdot (\ddot{\gamma}'-[(1-\alpha\varepsilon)\lambda'+\alpha \nabla\phi(\gamma)\cdot\gamma']\nabla\phi(\gamma)-[(1-\alpha\varepsilon)\lambda +\alpha\phi(\gamma)]D^2\phi(\gamma)\gamma')~dt\\
    & = 2 \int_0^1-\varepsilon\lambda'^2+\lambda'\nabla\phi(\gamma)\cdot\gamma'~dt \\ 
    & \qquad + 2 \mu\int_0^1\gamma' \cdot \ddot{\gamma}'-[(1-\alpha\varepsilon)\lambda'+\alpha \nabla\phi(\gamma)\cdot\gamma']\nabla\phi(\gamma)\cdot \gamma' \\
    & \qquad -[(1-\alpha\varepsilon)\lambda +\alpha\phi(\gamma)]\gamma'\cdot D^2\phi(\gamma)\gamma'~dt\\
    & = 2 \int_0^1-\varepsilon\lambda'^2+\lambda'\nabla\phi(\gamma)\cdot\gamma'~dt + 2 \mu\int_0^1\gamma' \cdot \ddot{\gamma}'-\alpha |\nabla\phi(\gamma)|^2|\gamma'|^2~dt \\ 
    & \qquad + 2 \mu\int_0^1 -(1-\alpha\varepsilon)\lambda'\nabla\phi(\gamma)\cdot \gamma'  -[(1-\alpha\varepsilon)\lambda +\alpha\phi(\gamma)]\gamma'\cdot D^2\phi(\gamma)\gamma'~dt\\
    & = 2 \int_0^1-\varepsilon\lambda'^2-\mu\alpha |\nabla\phi(\gamma)|^2|\gamma'|^2~dt - 2 \mu\int_0^1|\dot{\gamma}'|~dt+ 2\mu\gamma'\cdot\dot{\gamma}'|_0^1\\ 
    & \qquad + 2 \int_0^1 (1-\mu(1-\alpha\varepsilon))\lambda'\nabla\phi(\gamma)\cdot \gamma'-\mu(\lambda+\alpha\lambda')\gamma'\cdot D^2\phi(\gamma)\gamma'~dt.
\end{align*}
Notice that $\gamma(0, s) = p$ and $\gamma(1, s) = q$ for all $s$, so $\gamma'$ is zero at $t = 0$ and $t = 1$. Therefore, $\mu\gamma' \cdot \dot{\gamma}'\big{|}_0^1 = 0$. 
Thus,
\begin{align*}
    \frac{d}{ds} J & = 2 \int_0^1-\varepsilon\lambda'^2-\mu\alpha |\nabla\phi(\gamma)|^2|\gamma'|^2~dt - 2 \mu\int_0^1|\dot{\gamma}'|~dt\\ 
    &+ 2 \int_0^1 (1-\mu(1-\alpha\varepsilon))\lambda'\nabla\phi(\gamma)\cdot \gamma'-\mu(\lambda+\alpha\lambda')\gamma' \cdot D^2\phi(\gamma)\gamma'~dt.
\end{align*} 
Therefore, if $0<\alpha\varepsilon < 1$, then we can choose $\mu = \frac{1}{(1-\alpha\varepsilon)}$. Then, 
\begin{align*}
    \frac{d}{ds} J & = 2 \int_0^1-\varepsilon\lambda'^2-\mu\alpha |\nabla\phi(\gamma)|^2|\gamma'|^2~dt - 2 \mu\int_0^1|\dot{\gamma}'|+(\lambda+\alpha\lambda')\gamma' \cdot D^2\phi(\gamma)\gamma'~dt \\
    & \leq -2 \varepsilon \int_0^1\lambda'^2dt -2\mu\alpha  \int_0^1 |\nabla \phi|^2 |\gamma'|^2~dt + 2\mu \int_0^1 
    |\lambda +\alpha \lambda'| \|D^2\phi(\gamma)\||\gamma'|^2 dt. 
\end{align*}
We now need to bound the last term. Observe that 
\begin{align*}
    \lambda(t,s) &= \lambda(t,0)e^{-\varepsilon s} + \int_0^s\phi(\gamma)e^{\varepsilon \tau}~d\tau \cdot e^{-\varepsilon s}.
 \end{align*} 
 This gives 
 \begin{align*} 
   |\lambda(t,s)|  
    & \leq |\lambda(t,0)|e^{-\varepsilon s} +\max_{[0, s]} |\phi(\gamma)|e^{-\varepsilon s} \frac{1}{\varepsilon}(e^{\varepsilon s} - 1)\\
    & \leq |\lambda(t,0)|e^{-\varepsilon s} +\frac{ \max_{[0, s]} |\phi(\gamma)|}{\varepsilon}. 
\end{align*}
Therefore, using $\epsilon \alpha <1$,  
\begin{align*}
    |\lambda + \alpha \lambda'| &= |(1-\alpha\varepsilon)\lambda +\alpha\phi(\gamma)|\\
    &\leq (1-\alpha\varepsilon) |\lambda(t,0)|e^{-\varepsilon s}+ \frac{1-\alpha\varepsilon}{\varepsilon} \max_{[0, s]} |\phi(\gamma)|+\alpha|\phi(\gamma)|\\
    & = (1-\alpha\varepsilon) |\lambda(t,0)|e^{-\varepsilon s} +\frac{\max_{[0, s]} |\phi(\gamma)|}{\varepsilon}.
\end{align*}
Hence, we have 
\begin{align*}
    |\lambda +\alpha \lambda'|\|D^2 \phi(\gamma)\|  & \leq (1-\alpha\varepsilon) |\lambda(t,0)|e^{-\varepsilon s} \|D^2 \phi(\gamma)\| + \frac{\max_{[0, s]} |\phi(\gamma)|}{\varepsilon}   
    \|D^2 \phi(\gamma)\| \\
    & \leq (1-\alpha\varepsilon) |\lambda(t,0)|e^{-\varepsilon s} \|D^2 \phi(\gamma)\| + \frac{a}{\varepsilon} \|D^2 \phi(\gamma)\|.
\end{align*} 
Thus, the last term can be bounded
\begin{align*}
    2\mu\int_0^1 |\lambda + \alpha \lambda'|\|D^2 \phi(\gamma) \| |\gamma'|^2 ~dt & \leq  2\mu  (1-\alpha\varepsilon) \max_t |\lambda(t,0)|e^{-\varepsilon s} \max_t \|D^2 \phi(\gamma)||\int_0^1 |\gamma'|^2 dt \\
    &  + \frac{2a\mu}{\varepsilon} \max_t \|D^2 \phi(\gamma)||  \int_0^1  |\gamma'|^2~dt\\
    & \leq \left( \delta +  \frac{2a \mu}{\varepsilon} \max_t \|D^2 \phi(\gamma)||   \right)  \int_0^1 |\gamma'|^2 dt,
\end{align*}
where we have used the fact that $|O(1)e^{-\varepsilon s}| \leq \delta$ for $s>S$ when $S$ is large enough. Thus,
\begin{align*}
    \frac{d}{ds} J & 
    \leq  -2 \varepsilon \int_0^1\lambda'^2dt -2\mu\alpha  \int_0^1 |\nabla \phi(\gamma)|^2 |\gamma'|^2~dt + 2\mu \int_0^1 
    |\lambda +\alpha \lambda'| |D^2\phi(\gamma)||\gamma'|^2 dt \\
    & \leq  -2\varepsilon \int_0^1\lambda'^2 
    - (  2 \mu \alpha \nu - \delta - \frac{\mu \nu}{\varepsilon} )  
    \int_0^1 |\gamma'|^2~dt. 
\end{align*}
Note that we can choose $\varepsilon$ and $\alpha$ so that 
$\varepsilon \alpha >1/2$, and take $S$ suitably large  so that 
$$
\delta \leq  \delta_0:= 
(\alpha \varepsilon -1/2)\frac{\mu \nu}{\varepsilon}=\frac{\nu(\alpha \varepsilon -1/2)}{\varepsilon (1- \alpha \varepsilon)}.
$$
Hence 
$$
J' \leq -\beta J,  
$$
where $\beta = \min\{2\varepsilon, \delta_0 \}$. Thus,
$$
J(U(s))\leq J(U(S))e^{-\beta (s-S)}, \quad \forall s>S.
$$
\end{proof} 
\begin{remark}
Assumption A may appear unconventional, it provides  a reasonable constraint that keeps $\gamma$ close to the underlying surface. 
{ For instance with $\phi(x)=\frac{1}{2}(|x|^2-1)$, we have 
$$
\nabla \phi(x)=x, \quad D^2\phi(x)=I.
$$
To satisfy assumption A  
in the domain $\Sigma=\{x| |\phi(x)|\leq a\}=\{x|\quad 1-2a \leq |x|^2 \leq 1+2a \}$, it suffices to set 
$\nu=1-2a$ with $a$
satisfying $2a \leq \nu$. This implies $a \leq 1/4$ and $\nu \geq 1/2.$ 
We can also consider $\phi(x)$ as a signed distance function, which,  in this case, takes  the form:  
$$
\phi(x)=1- |x|.
$$
In the domain $\Sigma=\{x| \; |\phi(x)|\leq a\}=\{x|\quad 1-a\leq |x|\leq 1+a\}$, both 
$$
\nabla \phi(x)=-\frac{x}{|x|}, \quad \text{and} \quad  D^2\phi(x)=\frac{1}{|x|^3}\left[
x\otimes x -|x|^2I
\right] 
$$
 satisfy Assumption A with $\nu=1$, provided   $\frac{2a}{1-a}\leq 1$, which requires $0< a\leq \frac{1}{3}$. Note that for the signed distance function, the condition $|\nabla \phi|=1$ always holds. This property also explains why the signed distance function is highly effective in various  applications. 
 }
\end{remark}

\subsection{Discrete System}
  Note that the update described in (\ref{eqs:PDHG_update}) can be reformulated as 
   \begin{align*}
        \frac{\lambda_{k+1}-\lambda_k}{\tau_{\lambda}} & = -\frac{\varepsilon}{1+\varepsilon\tau_{\lambda}}\lambda_k + \frac{1}{1+\varepsilon\tau_{\lambda}}\phi(\gamma_k)\\
        \frac{\gamma_{k+1}-\gamma_k}{\tau_{\gamma}} & = \ddot{\gamma}_k(t)-\nabla\phi(\gamma_k)\cdot\left[ \left(1-\frac{1+w}{1+\varepsilon \tau_{\lambda}}\varepsilon\tau_{\lambda}\right)\lambda_k+\frac{(1+w)\tau_{\lambda}}{1+\varepsilon\tau_{\lambda}}\phi(\gamma_k) \right],
    \end{align*}   
for which the high-resolution PDE system (\ref{eq:equi0}) serves as a more accurate time-continuous counterpart. By some standard  truncation error analysis we can conclude the following:  
\begin{theorem} \label{thm7}
 Denote $U_k(t) =(\lambda_k(t), \gamma_k(t))$  as the semi-discrete update  defined in (\ref{eqs:PDHG_update}),  and $U(s, t)=(\lambda(s, t),\gamma(s, t))$  as the solution provided in Theorem \ref{thm5} with initial data $U(0, t)=U_0(t)$. Then for $\tau=\tau_\gamma =\frac{\tau_\lambda}{1+ \varepsilon \tau_{\lambda}}$, we have:
$$
\lim_{\tau \to 0} \sup \max_{0\leq k\leq S/\tau}\| U_k(t)-U(k\tau, t)\|=0
$$
for any fixed $S>0$. 
\end{theorem}  
The key question is: can we relate  convergence properties of high-resolution PDEs to those of  the corresponding discrete update scheme. The answer is yes, as Theorem \ref{thm7} holds true for $S=\infty$ under the structural assumptions on $\phi$. A 
full proof would involve carefully estimating a discrete Lyapunov functional such as 
$$
J_k=\int_0^1 (\lambda_{k+1}-\lambda_k)^2 dt +\mu \int_0^1 (\gamma_{k+1}-\gamma_k)^2 dt.
$$
Rather than going deeper with technical estimates, we now turn our attention to how the scheme behaves under discretization in the parameter $t$, and how it performs numerically. This also brings up a subtle issue: in cases where the geodesic between states is not unique, what does convergence actually mean -- and how should the scheme behaves.

\subsection{Discussion on the case with non-unique geodesics}
When the geodesic curve between points $p$ and $q$ on the level set surface is not unique, the initialization of $\gamma$ can significantly influence  the algorithm's convergence. In such cases, we introduce randomness into the initialization of $\gamma$. Our numerical experiments indicate that this randomized  initialization improves 
the likelihood of convergence to a valid geodesic curve. Intuitively, this may be because the algorithm tends to closest to the initial curve. Therefore, the initial choice of $\gamma$ plays a key role in determining which geodesic the algorithm ultimately finds in non-unique scenarios.

\section{Numerical Results} 
\subsection{Initialization} There are several ways to initialize the iteration algorithm. One straightforward approach is to set  $\gamma$ as a straight line connecting  $p$ to $q$, with $\lambda$ initialized to $0$. This provides a simple initial guess 
\[\begin{cases}
    \gamma_0(t) = p(1-t)+qt & \text{ such that } \gamma_0(0) = p, \gamma_0(1) = q,\\
    \lambda_0(t) = 0.
\end{cases}\]
When the geodesic curve is not unique, for instance, when $\gamma(0) = p$ and $\gamma(1) = q$ are antipodal points on the surface of a sphere, 
multiple geodesic curves exist between $p$ and $q$. To address this, we introduce randomness into the initialization of the curve $\gamma$ by selecting a random point $r$ on the surface and setting 
\[
\gamma_0(t) = p(1-t) +  \tau_r r(1-t)t + qt,
\]
where $\tau_r$ is a constant. In our numerical experiments, this  initialization method enables the algorithm to converge to a geodesic curve, even when $p$ and $q$ are antipodal. 
Since determining the uniqueness of a geodesic between points $p$ and $q$ can be challenging, this initialization approach is applicable for any pair of points. Figure \ref{fig:initialization} shows the resulting curve after $3500$ iterations,  when $\gamma$ is initialized to be a straight line (left) and when $\gamma$ is initialized with some randomness (right). Both results in Figure \ref{fig:initialization} are generated with the PDHG-inspired algorithm using regularization described by equation \ref{eqs:PDHG_update}. We use $\varepsilon = 0.01$ and $\omega = 1$.

\begin{figure}
    \centering
    \includegraphics[width = 0.45\linewidth]{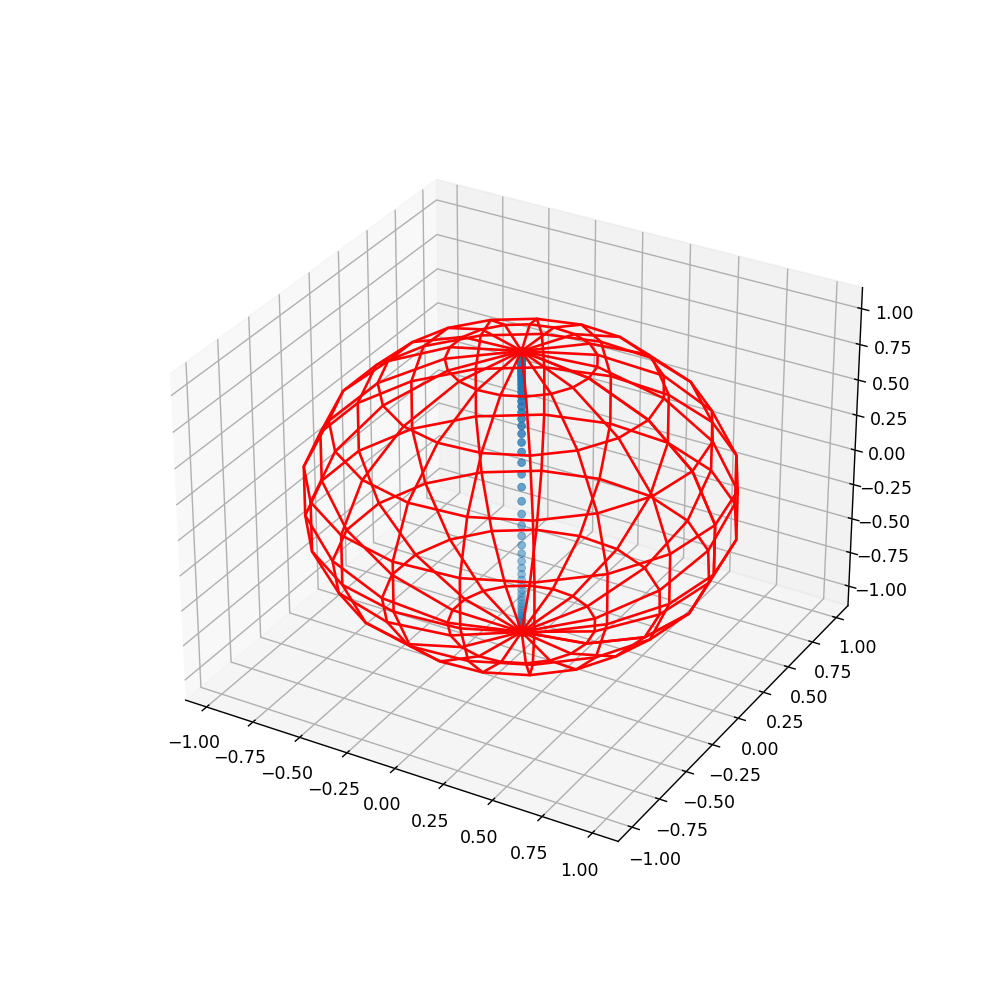}\includegraphics[width = 0.45\linewidth]{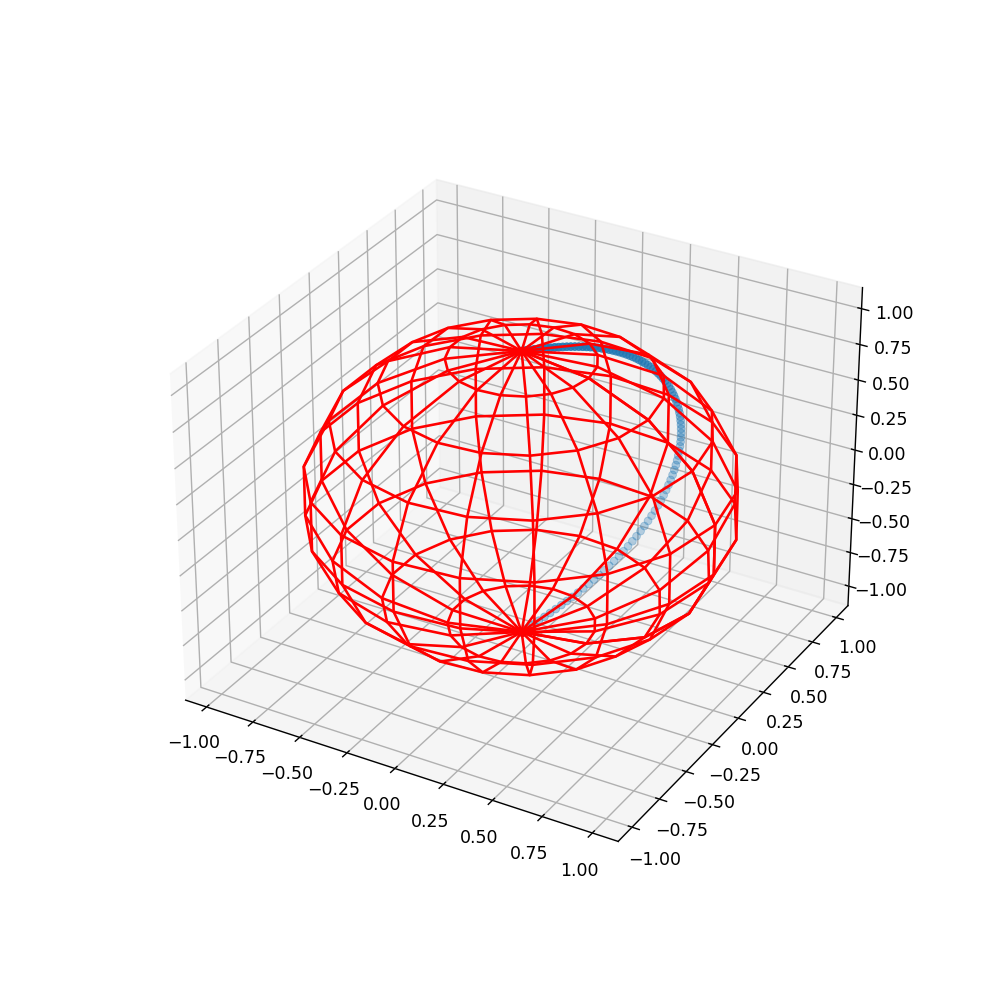}
    \caption{Left: Resulting curve $\gamma$ (in blue) after 3500 iterations of the algorithm to find the path between antipodal points $p$ and $q$, using a straight line as initial guess. Right: Resulting curve after 3500 iterations for the same antipodal points $p$ and $q$,  using random initialization.}
    \label{fig:initialization}
\end{figure}

\subsection{Numerical Solution}
To compute the solution of the marching scheme, we use finite differences. For some integer $m>0$, we define $\Delta t= 1/m$. For $i\in \{1,\ldots, m\}$, let $t_i = i\Delta t$ and for each $i\in \{1,\ldots , m-1\}$, we do 
\begin{equation}\label{eqs:Discrete_scheme}
\begin{cases}
\lambda_{k+1}(t_i) = \frac{1}{1+\varepsilon\tau_{\lambda}}\left(\lambda_k(t_i)+\tau_{\lambda}\phi(\gamma_{k}(t_i))\right), \\
\Tilde{\lambda}_{k+1}(t_i) = \lambda_{k+1}(t_i) + \omega(\lambda_{k+1}(t_i) - \lambda_k(t_i)),\\
    \gamma_{k+1}(t_i) = \gamma_k(t_i)-\tau_{\gamma}\left(-\frac{\gamma_k(t_{i+1})-2\gamma_k(t_i)+\gamma_k(t_{i-1})}{\Delta t^2}+\Tilde{\lambda}_{k+1}(t_i)\nabla\phi(\gamma_k(t_i))\right), \\
    \gamma_k(t_0)=p, \quad \gamma(t_m)=q.
\end{cases}
\end{equation}

To evaluate the results, two error metrics were used. The first, termed ``absolute error",  is defined as $||\gamma| - d|$, where $d$ represents the true geodesic distance. The second metric,  
referred to as ``surface error",  is given by 
\[ 
|\lambda^T\phi(\gamma)|,
\]
which reflects the extent to which the path $\gamma$ deviates from the surface. Surface error is a useful metric when the true geodesic distance can not be calculated. 

Table \ref{tab:error} shows the average absolute error, average relative error, and average computation time for the algorithm \ref{eqs:Discrete_scheme} for 10 pairs of points on the unit sphere. Here, $\tau_{\gamma} = 0.00001$, $\tau_{\lambda} = 0.7$, $\varepsilon = 0.01$, and $\omega = 1$ were used. 

\begin{table}[H]
    \centering
        \begin{tabular}{|p{3.99cm}|ccc|}
        \hline
         & Avg Absolute Error & Avg Relative Error & Avg Comp. Time\\
         \hline
         100 iterations & 0.169  &  9.08 \%  &  0.376 s \\
         \hline
         1000 iterations & 0.107  &  5.719 \%  &  4.175 s \\
         \hline
         2000 iterations  & 0.057  &  2.978 \%  &  8.279 s \\
         \hline
    \end{tabular}
   \caption{Average results for algorithm \ref{eqs:Discrete_scheme} for 10 pairs of points, 100 time points, $\tau_{\gamma} = 0.00001$, $\tau_{\lambda} = 0.7$, $\varepsilon = 0.01$, and $\omega = 1$}. 
    \label{tab:error}
\end{table} 

\subsection{Effect of Regularization Coefficient $\varepsilon$}
The choice of $\epsilon$  significantly affects both numerical stability and the quality of the computed geodesic. We tested a range of values (i.e., $\epsilon=10^{-4},...1$), as shown in Figure \ref{fig:vary_epsilon}. 
Small $\varepsilon$ leads to instability, while large $\varepsilon$ increases absolute error. 
\begin{figure}
    \begin{center}
    \includegraphics[height = 5.5cm]{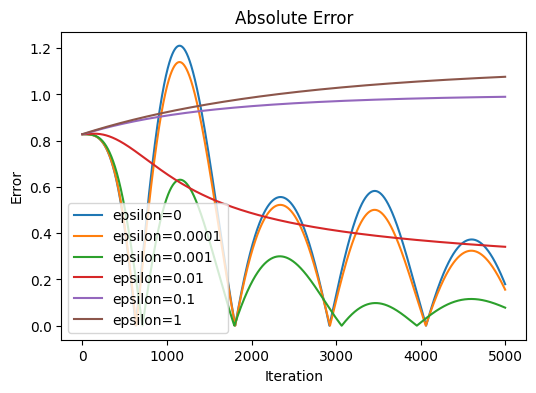}
    \includegraphics[height = 5.5cm]{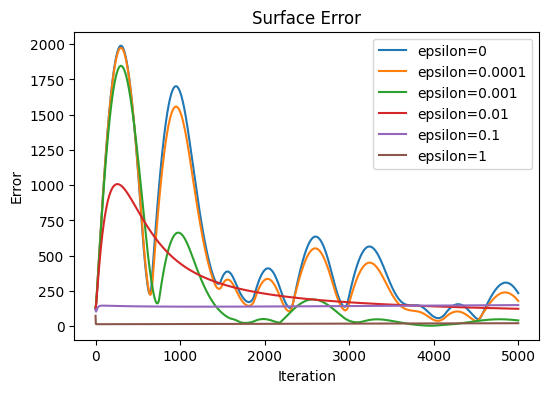}
\end{center}
        \caption{Plot of absolute error (left) and surface error (right) over iterations of the algorithm for finding a path $\gamma$ between antipodal points on a sphere. Each colored curve represents results for a different value of the regularization coefficient $\varepsilon$. } 
    \label{fig:vary_epsilon}
\end{figure}
Regularization is introduced to make sure the ascent step is well-defined.  Step size choices  alone are not enough to prevent $\lambda$ from drifting, and this can be seen in Figures \ref{fig:vary_steps_no_reg} and \ref{fig:vary_steps_reg}. Figure \ref{fig:vary_steps_no_reg} shows that, without regularization, all tested values of $\tau_{\lambda}$ lead to instability, as reflected in both absolute and surface error.
\begin{figure}
    \begin{center}
    \includegraphics[height = 5.5cm]{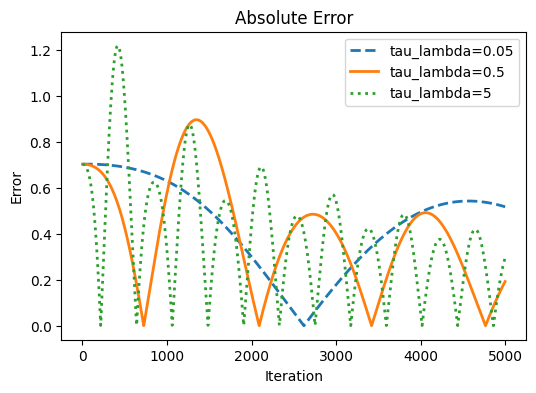}
    \includegraphics[height = 5.5cm]{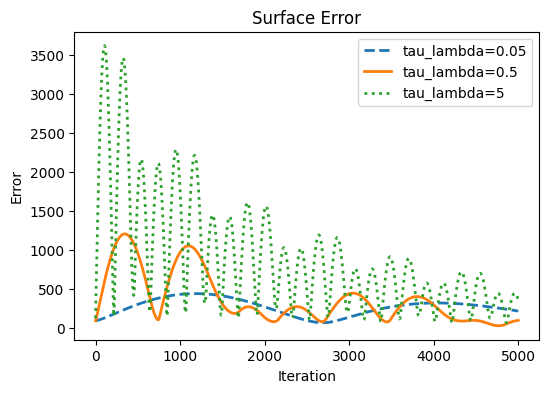}
\end{center}
    \caption{No regularization. Absolute error (left) and surface error (right) over iterations of the algorithm for finding a path $\gamma$ between antipodal points on the sphere. Different colors indicate different values of $\tau_{\lambda}$.} 
    \label{fig:vary_steps_no_reg}
\end{figure}

Figure \ref{fig:vary_steps_reg} shows absolute error and surface error for various $\tau_{\lambda}$ values with regularization ($\varepsilon = 0.01$). The results demonstrate that, with regularization, an appropriate choice of 
$\tau_{\lambda}$ yields stable and accurate behavior -- unlike the unregularized case in Figure \ref{fig:vary_steps_no_reg}. 
\begin{figure}
    \begin{center}
    \includegraphics[height = 5.5cm]{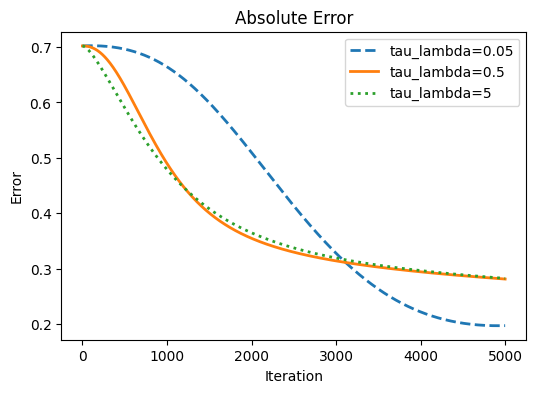}
    \includegraphics[height = 5.5cm]{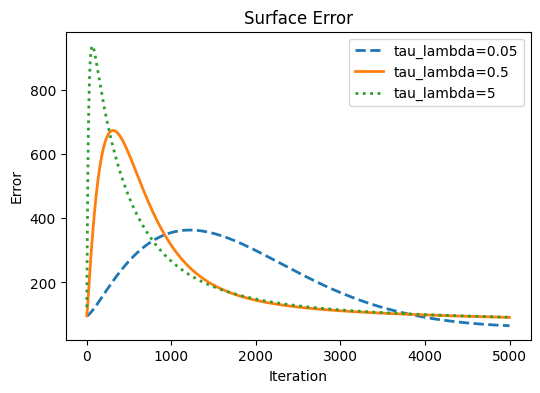}
\end{center}
    \caption{With regularization ($\varepsilon = 0.01$). Absolute error (left) and surface error (right) vs. iterations for computing a path $\gamma$ between antipodal points on the sphere. Different colors indicate different values of $\tau_{\lambda}$. }
    \label{fig:vary_steps_reg}
\end{figure}

\subsection{Examples}\label{sec:examples}
The algorithm \ref{eqs:Discrete_scheme} was used to approximate geodesic distances between points on several surfaces represented in $\R^3$.
Figure \ref{fig:sphere} shows the absolute error and surface error with respect to iterations of the algorithm when approximating the geodesic distance between antipodal points $p = [0,0,1]$ and $q = [0,0,-1]$. The final path obtained after 5000 iterations is plotted on a mesh sphere. For this example, the final absolute error was 0.265 and the final relative error was 8.445\%.

\begin{figure}[h]
    \centering
     \begin{minipage}{0.4\linewidth}
     \includegraphics[height = 5cm]   {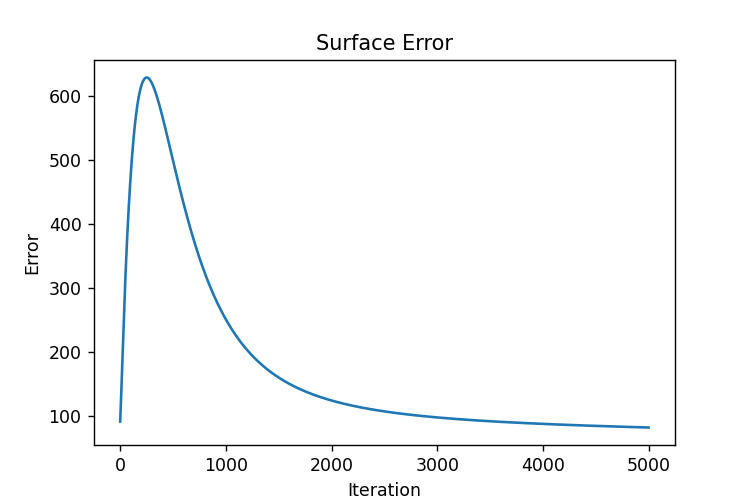}\\
        \includegraphics[height = 5cm]   {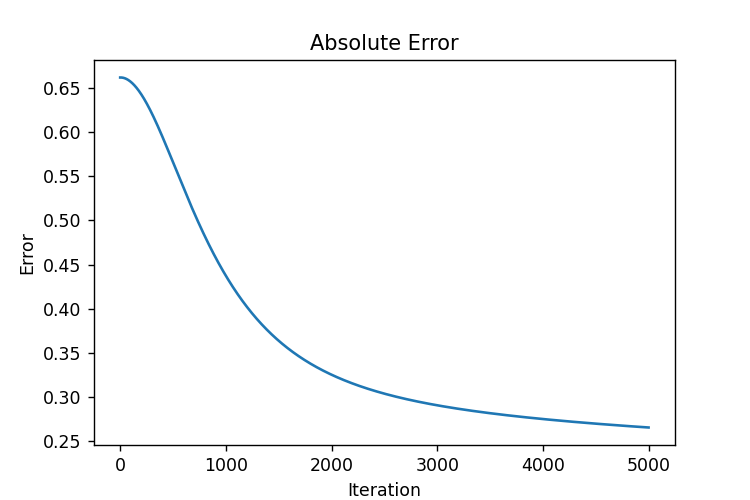} 
    \end{minipage}
    \begin{minipage}{0.55\linewidth}
        \begin{center}
            \includegraphics[height = 8cm]{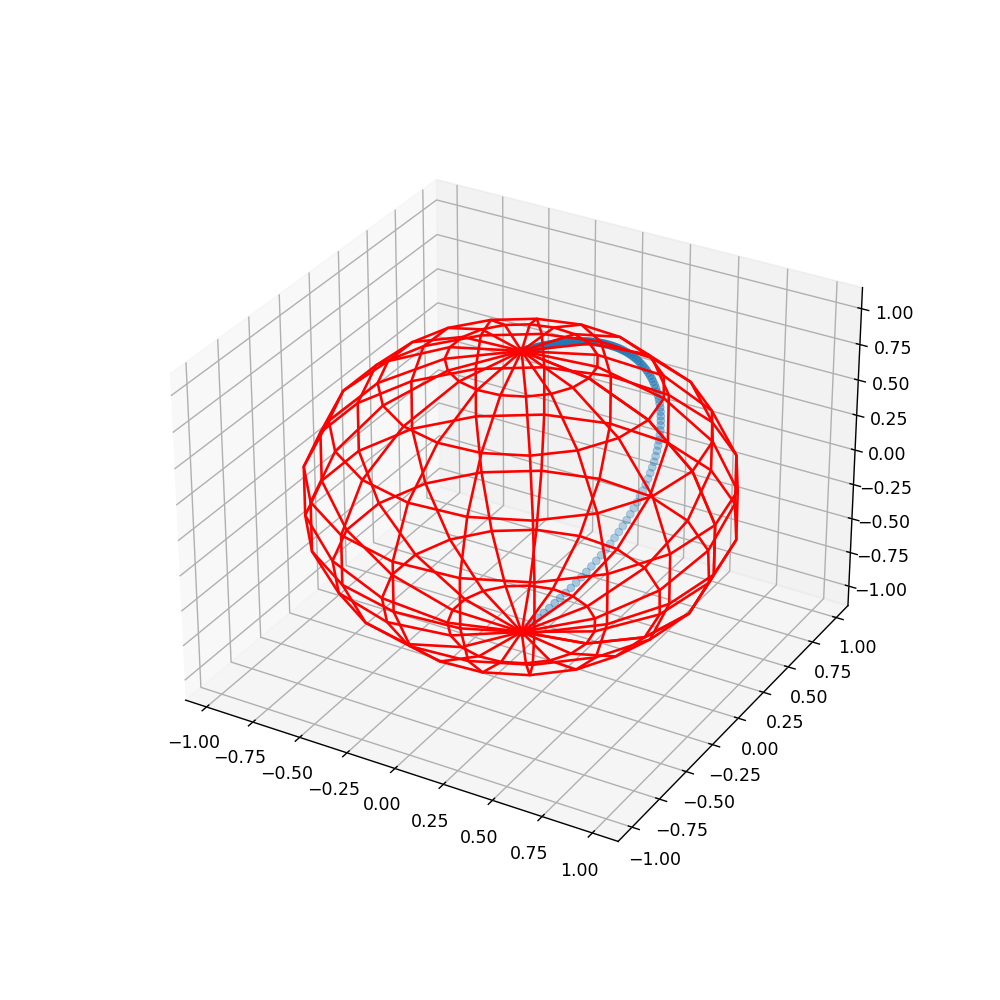} 
    \end{center}
    \end{minipage}
    \caption{The plots on the left shows the absolute error and surface error with respect to number of iterations of the algorithm \ref{eqs:Discrete_scheme}, where $\varepsilon = 0.01$ and $\omega = 1$. On the right, the approximated path $\gamma$ after 5000 iterations is plotted on a mesh unit sphere. For this example, the final absolute error was 0.265 and the final relative error was 8.445\%.}
\label{fig:sphere}
\end{figure}

Figure \ref{fig:torus} shows the surface error against iterations of the algorithm \ref{eqs:Discrete_scheme} when approximating the geodesic distance between two points on a torus. The approximated path after 5000 iterations is plotted on a mesh torus.
    \begin{figure}[h]
        \centering
         \begin{minipage}{0.4\linewidth}
        \includegraphics[height = 5cm]{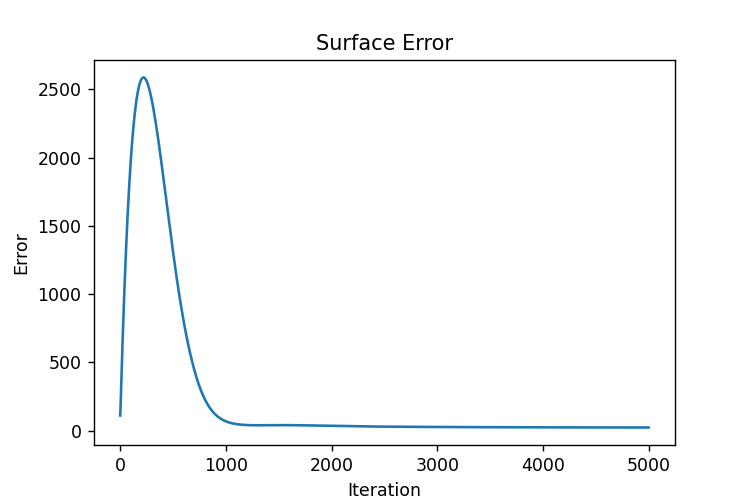} 
    \end{minipage}
    \begin{minipage}{0.55\linewidth}
        \begin{center}
            \includegraphics[height = 8cm]{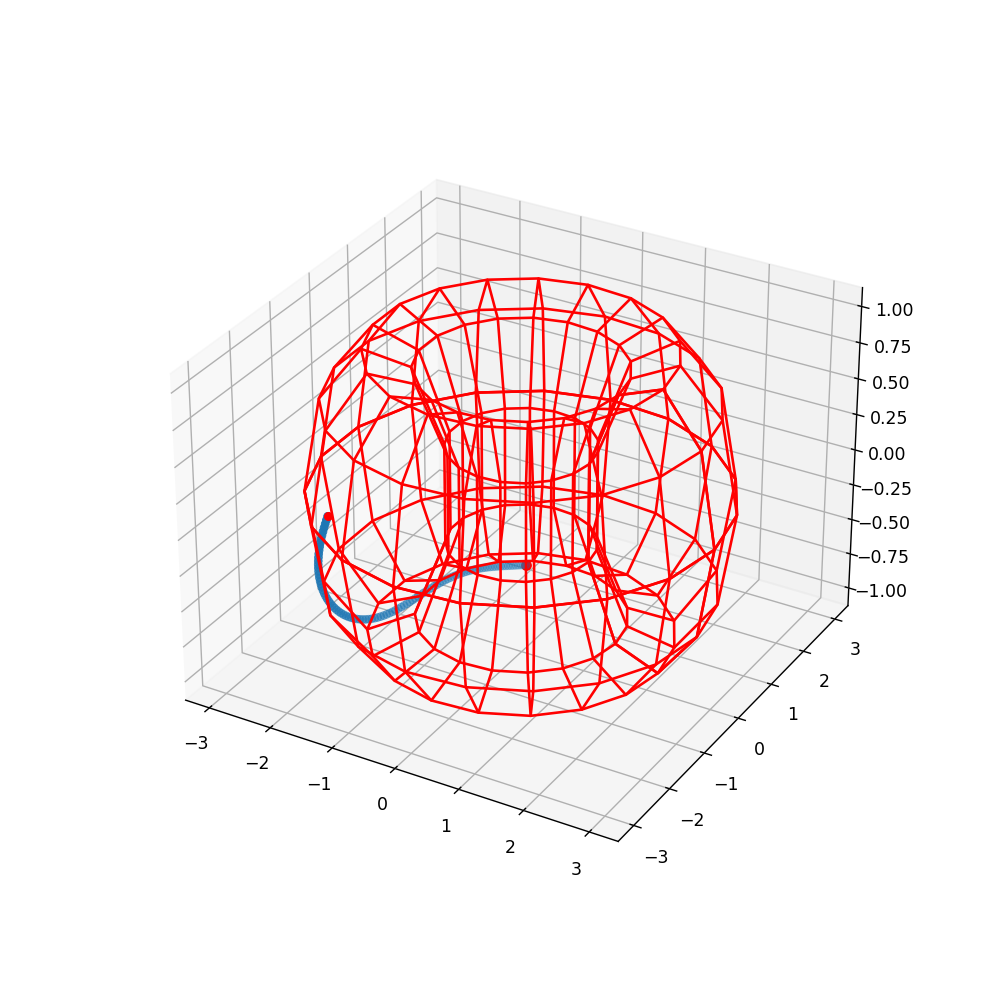} 
    \end{center}
    \end{minipage}
        \caption{The plot on the left shows the surface error with respect to number of iterations of algorithm \ref{eqs:Discrete_scheme} when used to approximate the geodesic distance between two points on a torus. On the right, the approximated path gamma after 5000 iterations is plotted on a mesh torus.}
        \label{fig:torus}
    \end{figure}

\subsection{Variation of the base algorithm}
Note that the update in (\ref{eqs:PDHG_update}) represents a discretization of the continuous PDE system. Alternatively, we can employ different discretizations of the PDE system to create various updates. For instance,  we  update $\lambda$ one more time before updating $\gamma$, resulting in the following
\begin{equation} \label{eqs:PDHG_var1}
\begin{cases}
    \lambda_{k+1}(t) = \frac{1}{1+\varepsilon\tau_{\lambda}}(\lambda_k(t)+\tau_{\lambda}\phi(\gamma_{k}(t))), \\
    \Tilde{\lambda}_{k+1}(t) = \lambda_{k+1}(t)+ \omega(\lambda_{k+1}(t)-\lambda_k(t)),\\
    \Bar{\lambda}_{k+1}(t) = \frac{1}{1+\varepsilon\tau_{\lambda}}(\Tilde{\lambda}_{k+1}(t)+\tau_{\lambda}\phi(\gamma_{k}(t))), \\
    \gamma_{k+1}(t) = \gamma_k(t) - \tau_{\gamma}(-\ddot{\gamma}_k(t)+\Tilde{\lambda}_{k+1}(t)\nabla\phi(\gamma_k(t))).
\end{cases}
\end{equation}

In Figure \ref{fig:PDHG_var1}, the plot on the left shows the absolute error against iterations of the algorithm \ref{eqs:PDHG_var1} when used to approximate distance between antipodal points on the sphere, and the final path $\gamma$ is shown on the right. 

\begin{figure}[h]
    \centering
     \begin{minipage}{0.4\linewidth}
        \includegraphics[height = 5cm]   {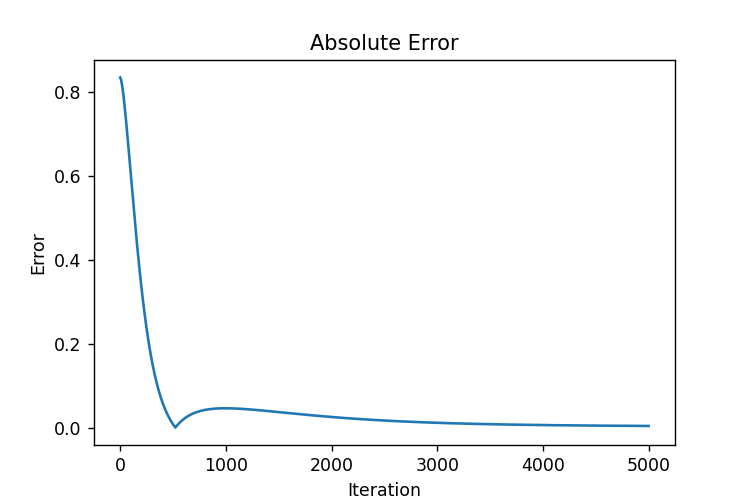} 
    \end{minipage}
    \begin{minipage}{0.55\linewidth}
        \begin{center}
            \includegraphics[height = 7cm]{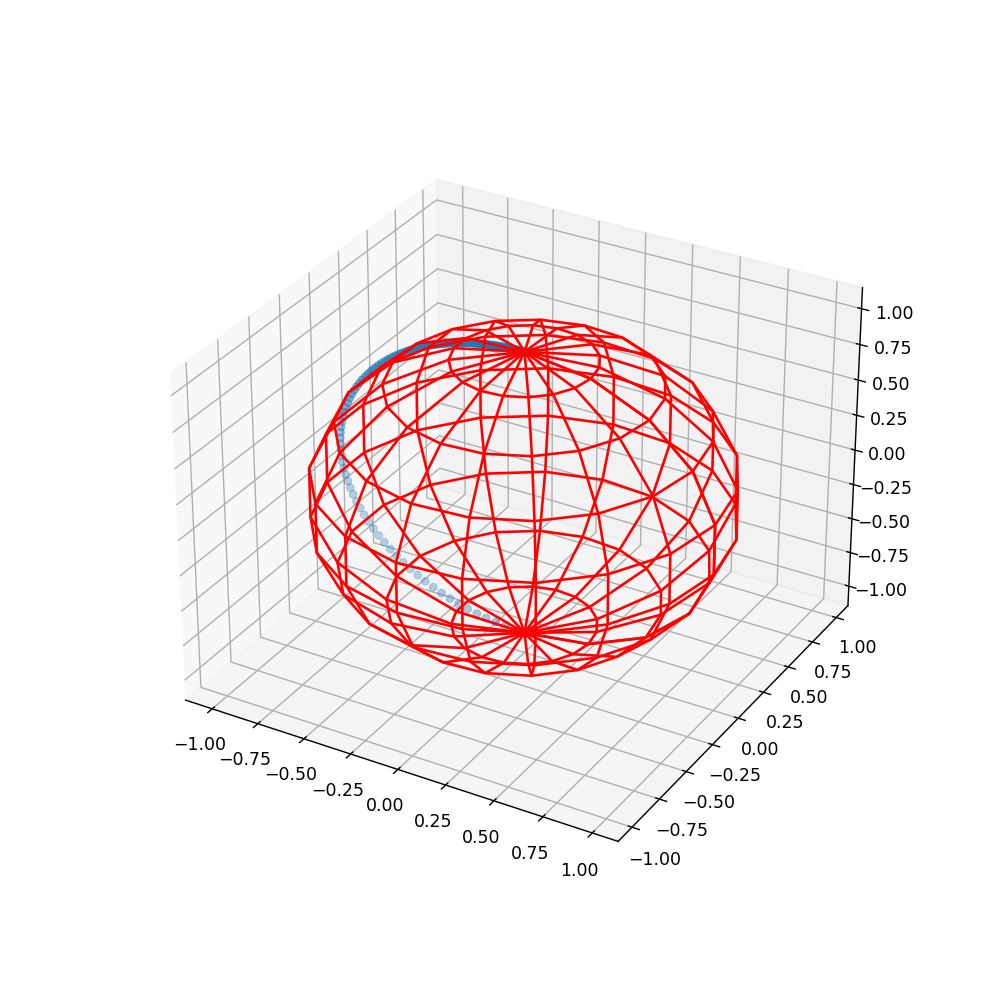} 
    \end{center}
    \end{minipage}
    \caption{The plot on the left shows the absolute error with respect to iterations of the algorithm \ref{eqs:PDHG_var1} described above, where $\varepsilon = 0.00001$ and $\omega = 1000$. The approximated path $\gamma$ after 5000 iterations is plotted on a mesh sphere on the right. For this example, the final absolute error was 0.004 and the final relative error was 0.118\%.}
    \label{fig:PDHG_var1}
\end{figure}

Figure \ref{fig:bunny} shows the surface error over time when a variation \ref{eqs:PDHG_var1} of the base algorithm is used to approximate geodesic distance between points on the surface of the ``Stanford bunny". For this surface, the signed distance function $\phi$ was approximated for a point $x$ by computing
\[\phi(x) = \min\{\|p-x\|~:~ p\in P\subset \Omega\},\]
where $P\subset \Omega$ is a set of points on the surface.
The approximated path after 100 iterations is plotted on a reconstruction of the surface of the Stanford bunny. 
\begin{figure}[h]
    \centering
        \begin{minipage}{0.4\linewidth}
        \includegraphics[height = 4.5cm]{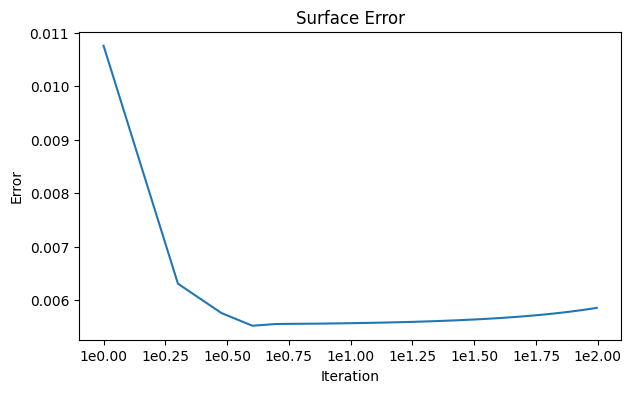} 
    \end{minipage}
    \begin{minipage}{0.55\linewidth}
        \begin{center}
            \includegraphics[height = 7.5cm]{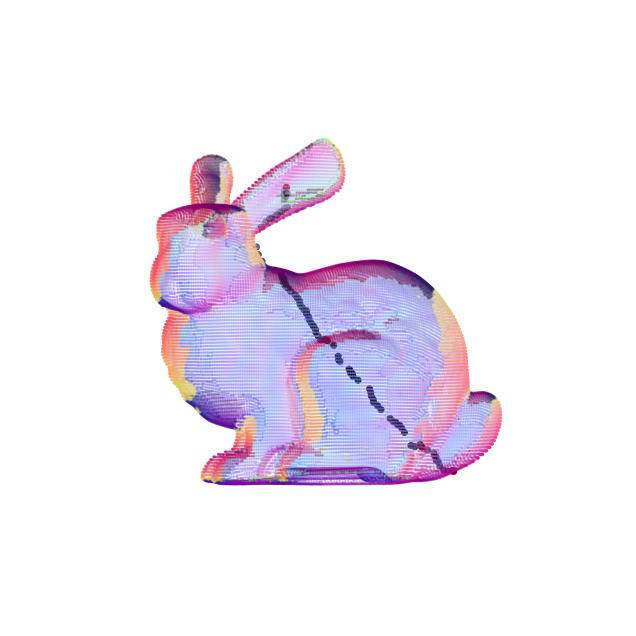} 
    \end{center}
    \end{minipage}
    \caption{The figure shows the surface error and final gamma plot for a path on the surface of the stanford bunny after 100 iterations. These plots were created using the update scheme \ref{eqs:PDHG_var1} with $\varepsilon = 0.0000001$ and $\omega = 100000$. Note that a log scale is used for the $x$-axis for the surface error plot.}
    \label{fig:bunny}
\end{figure}

Another algorithm, also grounded in the PDE formulation, showed promising performance: 
\begin{equation}\label{eqs:PDHG_var2}\begin{cases}
    \lambda_{k+1}(t) = \frac{1}{1+\varepsilon\tau_{\lambda}}(\lambda_k(t)+\tau_{\lambda}\phi(\gamma_{k}(t))), \\
    \Tilde{\lambda}_{k+1}(t) = (1-\alpha \varepsilon)\lambda_{k+1}(t)+ \alpha \phi(\gamma_k(t)),\\
    \gamma_{k+1}(t) = \gamma_k(t) - \tau_{\gamma}(-\ddot{\gamma}_k(t)+\Tilde{\lambda}_{k+1}(t)\nabla\phi(\gamma_k(t))).
\end{cases}\end{equation}
This scheme is a direct Euler discretization of the PDE system, where $\lambda'$ and $\gamma'$ are replaced by finite differences with time step 
$\frac{\tau_\lambda }{1+\varepsilon\tau_{\lambda}}$ and $\tau_\gamma$, respectively.  This is in sharp contrast to  scheme (\ref{eqs:PDHG_update}), where 
$$
 \Tilde{\lambda}_{k+1}(t) = \lambda_{k+1}(t)+ \omega(\lambda_{k+1}(t)-\lambda_k(t)).
$$
In algorithm \ref{eqs:PDHG_var2}, we can rewrite $\Tilde{\lambda}_{k+1}$ in terms of $\omega$ instead of $\alpha$ based on the calculations in Section \ref{sec:PDEsystem}, where $\alpha = (1+\omega)\frac{\tau_{\lambda}}{1+\varepsilon\tau_{\lambda}}$, 
\[\Tilde{\lambda}_{k+1} = \lambda_{k+1} + \frac{\alpha}{\tau_{\lambda}}(\lambda_{k+1}-\lambda_k) = \lambda_{k+1} +\frac{(1+\omega)}{1+\varepsilon\tau_{\lambda}}(\lambda_{k+1}-\lambda_k).\]

In Figure \ref{fig:PDHG_var2}, the plot on the left shows the absolute error against iterations of the algorithm \ref{eqs:PDHG_var2} when used to approximate distance between antipodal points on the sphere, and the final path $\gamma$ is shown on the right. 

\begin{figure}[h]
    \centering
     \begin{minipage}{0.4\linewidth}
        \includegraphics[height = 5cm]   {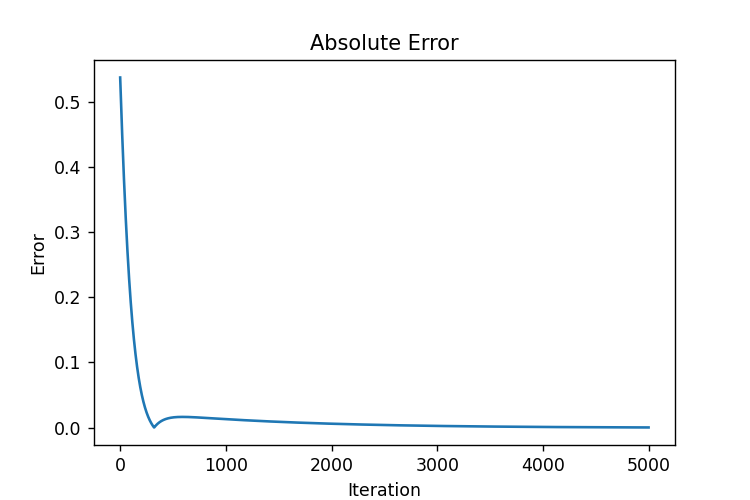} 
    \end{minipage}
    \begin{minipage}{0.55\linewidth}
        \begin{center}
            \includegraphics[height = 7cm]{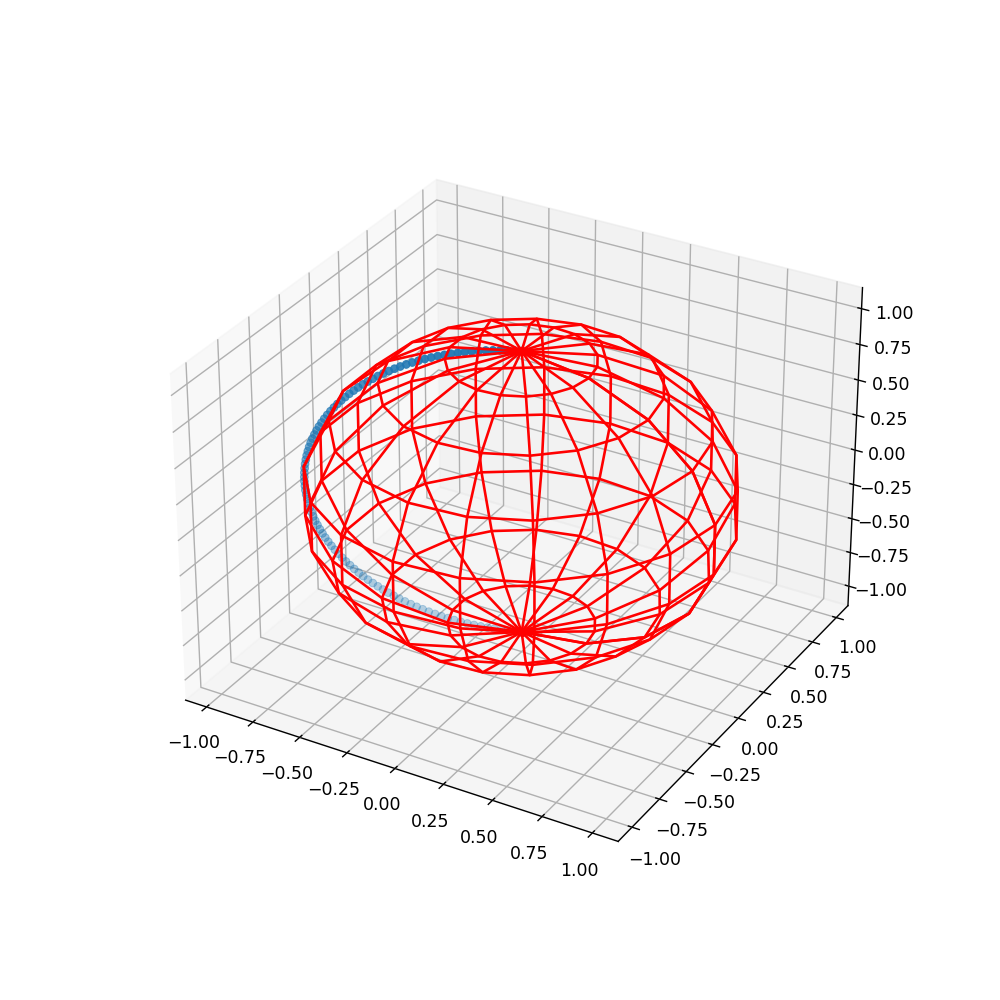} 
    \end{center}
    \end{minipage}
    \caption{The plot on the left shows the absolute error with respect to iterations of the algorithm \ref{eqs:PDHG_var2} described above, where $\varepsilon = 0.00001$ and $\alpha = 1000$. The approximated path $\gamma$ after 5000 iterations is plotted on a mesh sphere on the right. For this example, the final absolute error was 0.001, and the final relative error was 0.005\%.}
    \label{fig:PDHG_var2}
\end{figure}

\begin{table}[]
    \centering
    \begin{tabular}{|c|cc|}
    \hline
        Scheme & Absolute Error & Relative Error\\
        \hline
        PDHG \ref{eqs:PDHG_update} & 0.175 & 5.567\% \\
        \hline
        PDHG var 1 \ref{eqs:PDHG_var1} & 0.004 & 0.118\%  \\
        \hline
        PDHG var 2 \ref{eqs:PDHG_var2} & 0.0001 & 0.005\% \\
        \hline
    \end{tabular}
    \caption{Absolute and relative errors after 5000 iterations of the base algorithm and the two variations of the base algorithm when used to approximate geodesic distance between antipodal points on the unit sphere. For all schemes, $\varepsilon = 0.0001$,  $\omega = 1000$, and $\alpha = 1000$. }
    \label{tab:error_vars}
\end{table}

\section{{Convergence rate for the planar surface case}}\label{sec5}
Consider the case where $\phi(\gamma) = a\cdot \gamma$, with $a\in \R^3$. Then the variational derivatives of the Lagrangian are
\begin{align*}
    \frac{\delta}{\delta\gamma}\mathcal{L}_{\varepsilon}[\gamma, \lambda] = -\ddot{\gamma} + \lambda a, \quad 
    \frac{\delta}{\delta\lambda}\mathcal{L}_{\varepsilon}[\gamma, \lambda]   = a\cdot \gamma - \varepsilon\lambda.
\end{align*}
The PDHG scheme becomes 
\[\begin{cases}
    \lambda_{k+1} & = \lambda_k + \tau_{\lambda}(a\cdot\gamma_k - \varepsilon\lambda_{k+1}), \\
    \Tilde{\lambda}_{k+1} & = \lambda_{k+1} + \omega(\lambda_{k+1}- \lambda_k), \\
     \gamma_{k+1} & = \gamma_k -\tau_{\gamma}(-\ddot{\gamma}_{k+1}+ \Tilde{\lambda}_{k+1}a).
\end{cases}\]
Setting $\omega = 1$, the update for $\gamma$ simplifies to 
\[(1-\tau_{\gamma}\partial^2)\gamma_{k+1} = \gamma_k - \tau_{\gamma}(2\lambda_{k+1}-\lambda_k)a.\]
Solving the 2nd order equation with zero boundary conditions yields
\begin{align*}
\gamma_{k+1} &= (1-\tau_{\gamma}\partial^2)^{-1}(\gamma_k - \tau_{\gamma}(2\lambda_{k+1}-\lambda_k)a)\\
& =\int_0^1 G(t, s)\left( \gamma_k(s) - \tau_{\gamma}(2\lambda_{k+1}(s)-\lambda_k(s) )a\right)ds,
\end{align*}
where $G$ is the Green function defined by 
\[
G(t,s) = \frac{\sqrt{\tau_\gamma}}{\sinh (1/\sqrt{\tau_\gamma})} 
\begin{cases}
\sinh((1 - s)/\sqrt{\tau_\gamma}) \sinh(t/\sqrt{\tau_\gamma}), & 0 \leq t \leq s, \\
\sinh(s/\sqrt{\tau_\gamma}) \sinh((1 - t)/\sqrt{\tau_\gamma}), & s < t \leq 1.
\end{cases}
\]
We can prove the following result.
\begin{theorem}
    Assume $\phi(\gamma) = a\cdot \gamma$ for some $a \in \R^3$, and let $\lambda_k$ and $\gamma_k$ be the iterates generated by the  scheme
    \[\begin{cases}
        \lambda_{k+1} & = \lambda_k + \tau_{\lambda}(a\cdot\gamma_k - \varepsilon\lambda_{k+1})\\
        \gamma_{k+1} &= (1-\tau_{\gamma}\partial^2)^{-1}(\gamma_k - \tau_{\gamma}(2\lambda_{k+1}-\lambda_k)a).
    \end{cases}\]
    Denote the matrix and variables 
    $$
    A = \begin{bmatrix}
    \frac{1}{\tau_{\lambda}} & a^T\\
    a & \frac{1}{\tau_{\gamma}}I_{3\times 3}
\end{bmatrix}, \quad \xi = \begin{bmatrix}
    \lambda\\ \gamma
\end{bmatrix}, \;  \bar{\xi}_k = \begin{bmatrix}
    \bar{\lambda}_k\\ \bar{\gamma}_k
\end{bmatrix} = \frac{1}{k}\sum_{i=1}^k \xi_i.
$$ 
Then the system converges for any step sizes $\tau_{\lambda}, \tau_{\gamma}$ satisfying  
\[
\tau_{\lambda}\tau_{\gamma} < \frac{1}{|a|^2} = \frac{1}{a^Ta},
\]
with an ergodic convergence rate of $\mathcal{O}(1/k)$. Specifically, for any $k\geq 1$ and any $(\gamma, \lambda)$, 
\[\mathcal{L}_{\varepsilon}(\bar{\gamma}_k, \lambda) - \mathcal{L}_{\varepsilon}(\gamma, \bar{\lambda}_k) \leq \frac{1}{2k}\|\xi_0 - \xi\|^2_A. \]
\end{theorem}
\begin{proof}
From the update scheme, we have 
\[\begin{cases}
    \frac{\lambda_k -\lambda_{k+1}}{\tau_{\lambda}}+ a\cdot(\gamma_k - \gamma_{k+1}) &= \varepsilon\lambda_{k+1}-a\cdot \gamma_{k+1} = -\frac{\delta}{\delta\lambda}\mathcal{L}_{\varepsilon}^{k+1}\\
    \frac{\gamma_k - \gamma_{k+1}}{\tau_{\gamma}} + (\lambda_k-\lambda_{k+1}) a & = -\ddot{\gamma}_{k+1}+\lambda_{k+1}a = \frac{\delta}{\delta\gamma}\mathcal{L}_{\varepsilon}^{k+1}.
\end{cases}\]
Now consider the difference in Lagrangian values:
\begin{align*}
    \mathcal{L}_{\varepsilon}(\gamma_{k+1},\lambda) - \mathcal{L}_{\varepsilon}(\gamma,\lambda_{k+1}) & = \Delta_1 + \Delta_2 + \Delta_3.
\end{align*}
Term-by-term analysis gives:
\begin{itemize}
    \item Kinetic term difference:
    \begin{align*}
    \Delta_1 & = \frac{1}{2}\int_0^1\left(|\dot{\gamma}_{k+1}|^2 -|\dot{\gamma}|^2\right)~dt\\
    & = \int_0^1 (\gamma-\gamma_{k+1})\cdot\ddot{\gamma}_{k+1}~dt - \frac{1}{2}\int_0^1|\dot{\gamma}_{k+1}-\dot{\gamma}|^2~dt,
\end{align*}
\item By-linear term difference: 
\begin{align*}
    \Delta_2 & = \int_0^1 \lambda a\cdot \gamma_{k+1}~dt -\int_0^1 \lambda_{k+1}a\cdot \gamma~dt\\
     & = \int_0^1 (\lambda - \lambda_{k+1})a\cdot \gamma_{k+1}~dt + \int_0^1\lambda_{k+1}a\cdot (\gamma_{k+1}-\gamma)~dt,
\end{align*}
\item Regularization term difference: 
\begin{align*}
    \Delta_3 & = -\frac{\varepsilon}{2}\int_0^1 \lambda^2~dt + \frac{\varepsilon}{2}\int_0^1 \lambda_{k+1}^2~dt\\
    & = \int_0^1(\lambda_{k+1} - \lambda)\varepsilon\lambda_{k+1}~dt - \frac{\varepsilon}{2}\int_0^1 (\lambda- \lambda_{k+1})^2~dt.
\end{align*}
\end{itemize}
Combining all:
\begin{align*}
    & \mathcal{L}_{\varepsilon}(\gamma_{k+1},\lambda) - \mathcal{L}_{\varepsilon}(\gamma,\lambda_{k+1}) \\
    \qquad  = &\int_0^1 (\gamma - \gamma_{k+1})\cdot (\ddot{\gamma}_{k+1} - \lambda_{k+1}a)~dt +\int_0^1 (\lambda-\lambda_{k+1})(a\cdot \gamma_{k+1} - \varepsilon\lambda_{k+1})~dt \\
      \qquad  &- \frac{1}{2}\int_0^1 |\dot{\gamma} - \dot{\gamma}_{k+1}|^2~dt - \frac{\varepsilon}{2}\int_0^1 |\lambda - \lambda_{k+1}|^2~dt.
\end{align*}
Using the update expressions, 
\begin{align*}
     \leq & \int_0^1 (\gamma_{k+1}-\gamma)\cdot \left[\frac{\gamma_k -\gamma_{k+1}}{\tau_{\gamma}} + a(\lambda_k -\lambda_{k+1})\right]~dt \\
    &+ \int_0^1(\lambda_{k+1}- \lambda)\left[\frac{\lambda_k - \lambda_{k+1}}{\tau_{\lambda}}+ a\cdot(\gamma_k - \gamma_{k+1})\right]~dt.
  \end{align*}   
  Let $A = \begin{bmatrix}
    \frac{1}{\tau_{\lambda}} & a^T\\
    a & \frac{1}{\tau_{\gamma}}I_{3\times 3}
\end{bmatrix}$, and define $\xi = \begin{bmatrix}
    \lambda\\ \gamma
\end{bmatrix}$, then we obtain 
  \begin{align*} 
     =& \int_0^1 A(\xi_k - \xi_{k+1})\cdot (\xi_{k+1}-\xi)~dt \\
     = & \frac{1}{2}\int_0^1 A(\xi_k - \xi)\cdot (\xi_k - \xi)~dt - \frac{1}{2}\int_0^1A(\xi_{k+1} - \xi)\cdot (\xi_{k+1}- \xi)~dt\\ &- \frac{1}{2}\int_0^1 A(\xi_k - \xi_{k+1})\cdot (\xi_k- \xi_{k+1})~dt\\
     \leq &\frac{1}{2}\int_0^1 \|\xi_k - \xi\|_A^2~dt - \frac{1}{2}\int_0^1 \|\xi_{k+1}- \xi\|_A^2~dt.
\end{align*}
Hence, the algorithm converges if
\[\tau_{\lambda}\tau_{\gamma} < \frac{1}{|a|^2} = \frac{1}{a^Ta}.\] 
Further, define ergodic averages,    $\bar{\lambda}_k = \frac{1}{k}\sum_{i=1}^k \lambda_i$ and 
$\bar{\gamma}_k = \frac{1}{k}\sum_{i=1}^k \gamma_i$. 
By Jensen's inequality: 
\[\int_0^1 |\dot{\bar{\gamma}}_k|~dt  \leq \frac{1}{k}\sum_{i=1}^k\int_0^1 |\dot{\gamma_i}|~dt.\]
Using the telescoping sum, we conclude  
\begin{align*}
    \mathcal{L}_{\varepsilon}(\bar{\gamma}_k, \lambda) - \mathcal{L}_{\varepsilon}(\gamma, \bar{\lambda}_k) 
    & \leq \frac{1}{k}\sum_{i=0}^{k-1} (\mathcal{L}(\gamma_{i+1}, \lambda) -  \mathcal{L}(\gamma,\lambda_{i+1}))\\
     & \leq \frac{1}{2k}\|\xi_0 -  \xi\|^2_A. 
\end{align*}
Thus, the method achieves ergodic convergence at rate $O(1/k)$.
\end{proof}

\section{Discussion and future directions}  

In this work, we introduce a primal-dual level set method for computing geodesic distances. The underlying surface is represented as the zero level set of
a function, transforming the problem into a constrained minimization framework. Using  the primal-dual methodology, we incorporate  regularization and acceleration techniques to develop an efficient and robust algorithm. 

We analyze the long-term behavior of the high-resolution PDE system assicaited with the method by employing a Lyapunov function. Combined with structural conditions imposed on the level set function, this analysis establishes exponential convergence rates. 
Furthermore, relating  these results to discrete-time settings, we gain further  insights into the algorithm's performance. This proposed approach is computationally efficient, straightforward to implement, and offers several alternative extensions. Numerical examples illustrate that the method converges to geodesics as the level of refinement increases. 

Despite its strengths, our study has limitations that point to future research  directions. The convergence analysis relies on specific assumptions about the level set function. Relaxing this assumption to develop a more general framework would be a valuable extension of our work. Additionally, our current work is limited to static surfaces, whereas many practical 
applications require handling dynamically 
moving surfaces. Exploring the adaptability of our framework to moving surfaces  will be an important avenue for future work.

We can generalize the current framework to handle a general Lagrangian $L(\gamma, \dot\gamma)$ by solving the problem:  
\[
\min_{\Gamma(p, q, \Omega)} \int_0^1 L(\gamma, \dot\gamma)dt
\]
provided that $L$ is smooth and convex in its second argument. Using the saddle-point formulation 
\[
\inf_{\gamma}\sup_{\lambda} \left\{ \int_0^1 L(\gamma(t), \dot{\gamma}(t))~dt  + \int_0^1 \lambda(t)\phi(\gamma(t))~dt\right\}, 
\]
we derive with the same regularization the following iterative updates:   
\begin{equation*}\begin{cases}
    \lambda_{k+1}(t) = \frac{1}{1+\varepsilon\tau_{\lambda}}(\lambda_k(t)+\tau_{\lambda}\phi(\gamma_{k}(t))), \\
    \Tilde{\lambda}_{k+1}(t) = \lambda_{k+1}(t)+ \omega(\lambda_{k+1}(t)-\lambda_k(t)),\\
    \gamma_{k+1}(t) = \gamma_k(t) - \tau_{\gamma}(L_{\gamma_k} -\frac{d}{dt} L_{\dot{\gamma}_k} +\Tilde{\lambda}_{k+1}(t)\nabla\phi(\gamma_k(t))),
\end{cases}\end{equation*}
where $\varepsilon>0$ is a regularization parameter,  $\omega \in [0, 1]$ controls relaxation, and $\tau_\lambda, \tau_\gamma$ are  step sizes.


\bibliographystyle{siam} 
\bibliography{GeodesicApproximation}

\end{document}